\documentclass[a4paper, 11pt]{article}

\usepackage{amsthm}
\usepackage{amsmath}
\usepackage{dsfont}
\usepackage{amssymb}
\usepackage{booktabs}
\usepackage{graphicx}
\usepackage[round]{natbib}
\usepackage{bbm}
\usepackage{url}
\usepackage{pdflscape}
\usepackage{footmisc}
\usepackage{authblk}
\usepackage{blindtext}
\usepackage{placeins}
\usepackage[font=small,labelfont=bf]{caption}
\usepackage[colorinlistoftodos,prependcaption,textsize=scriptsize]{todonotes}

\usepackage[colorlinks,citecolor=blue, urlcolor=blue]{hyperref}
\usepackage[margin=2.5cm]{geometry}

\theoremstyle{plain}  
\newtheorem{theorem}{Theorem} 
\newtheorem{lemma}{Lemma} 
\newtheorem{proposition}{Proposition} 
\newtheorem{corollary}{Corollary} 
\newtheorem{assumption}{Assumption} 

\theoremstyle{definition}

\theoremstyle{remark}

\newcommand{\real}{\mathbb{R}}
\newcommand{\Prob}{\mathrm{pr}}

\begin{document}

\title{Honest calibration assessment for binary outcome predictions}

\author[1]{Timo Dimitriadis}
\author[2]{Lutz D\"umbgen}
\author[3]{Alexander Henzi}
\author[4]{Marius Puke}
\author[2]{Johanna Ziegel}

\affil[1]{Heidelberg University, Germany}
\affil[2]{University of Bern, Switzerland}
\affil[3]{ETH Zürich, Switzerland}
\affil[4]{University of Hohenheim, Germany}


\affil[ ]{\linebreak 
	\href{mailto:timo.dimitriadis@awi.uni-heidelberg.de}{timo.dimitriadis@awi.uni-heidelberg.de}, 
	\href{mailto:lutz.duembgen@stat.unibe.ch}{lutz.duembgen@stat.unibe.ch},
	\href{mailto:alexander.henzi@stat.math.ethz.ch}{alexander.henzi@stat.math.ethz.ch},
	\href{mailto:marius.puke@uni-hohenheim.de}{marius.puke@uni-hohenheim.de},
	\href{mailto:johanna.ziegel@stat.unibe.ch}{johanna.ziegel@stat.unibe.ch}
}

\maketitle

\begin{abstract}
	Probability predictions from binary regressions or machine learning methods ought to be calibrated: If an event is predicted to occur with probability $x$, it should materialize with approximately that frequency, which means that the so-called calibration curve $p(\cdot)$ should equal the identity, $p(x) = x$ for all $x$ in the unit interval.
	We propose honest calibration assessment based on novel confidence bands for the calibration curve, which are valid only subject to the natural assumption of isotonicity. 
	Besides testing the classical goodness-of-fit null hypothesis of perfect calibration, our bands facilitate inverted goodness-of-fit tests whose rejection allows for the sought-after conclusion of a sufficiently well specified model. 
	We show that our bands have a finite sample coverage guarantee, are narrower than existing approaches, and adapt to the local smoothness of the calibration curve $p$ and the local variance of the binary observations.
	In an application to model predictions of an infant having a low birth weight, the bounds give informative insights on model calibration. \\[0.8em]
\end{abstract}

\section{Introduction}
\label{sec:Introduction}

Consider first a univariate regression setting with fixed real covariates $x_1 \le \cdots \le x_n$ and independent binary observations $Y_1, \dots, Y_n \in \{0, 1\}$, where $\Prob(Y_i = 1) = p(x_i)$ for some unknown regression function $p : \real \to [0,1]$. Standard parametric models for this setting, e.g.\ logistic or probit regression, involve monotone regression functions $p$. Thus, an interesting nonparametric alternative would be to draw inference on $p$ under the sole assumption that it is isotonic on $\real$,
\begin{equation}
	\label{eq:isotonicity}
	p(x) \ \le \ p(x'), \quad x \leq x'.
\end{equation}

In the specific applications we have in mind, the $x_i$ are themselves probability predictions for the binary outcomes, i.e.~$x_i \in [0,1]$ is a prediction for the probability of the event $\{Y_i = 1\}$. In practice, the predictions can be obtained from a test sample of binary regressions, machine learning methods, or any other statistical model for binary data. A reliable interpretation of these predictions relies on the property of calibration, meaning that if the value $x_i$ is predicted, the corresponding event should indeed occur with probability $x_i$. In this setting, the regression function $p$ is called calibration curve, and it maps the predicted probabilities $x_i$ to the actual, or recalibrated, event probabilities $p(x_i) = \Prob(Y_i = 1)$. For calibrated predictions, the calibration curve equals the diagonal, $p(x) = x$ for all $x \in [0,1]$. Drawing inference about $p$ thus allows to assess the calibration of the predictions.

Testing the null hypothesis of calibration, $\mathbb{H}_0 \colon p(x) = x$ for all $x \in [0,1]$, is closely related to goodness-of-fit testing, which is crucial in applications, see e.g., \citet[Section 4.2]{Tutz2011} and \citet[Chapter 5]{Hosmer2013Book}.
It is still regularly carried out by the classical test of \cite{HosmerLemeshow1980}, which groups the predictions $x_i$ into bins and applies a $\chi^2$-test.
It is however subject to multiple criticisms:
First, its ad hoc choice of bins can result in untenable instabilities \citep{Bertolini2000, Allison2014}.
Second, placing the hypothesis of calibration in the null only allows for rejecting calibration rather than showing that a model is sufficiently well calibrated, where the latter would be highly desirable for applied researchers.
Third, the test rejects essentially all, even acceptably well-specified models in large samples \citep{Nattino2020, Paul2013}, resulting in calls for a goodness-of-fit tests with inverted hypotheses \citep{Nattino2020Rejoinder}, that is, tests where the hypothesis $p(x) = x$ is contained in the alternative.

We propose a statistically sound solution to these criticisms by constructing honest, simultaneous confidence bands $(L^\alpha,U^\alpha)$ for the function $p$. That is, for a given small number $\alpha \in (0,1)$ and $\mathcal{Y} := (Y_i)_{i=1}^n$, we compute data-dependent functions $L^\alpha = L^\alpha(\cdot,\mathcal{Y})$ and $U^\alpha = U^\alpha(\cdot,\mathcal{Y})$ on $\real$ such that
\begin{equation}
	\label{eq:coverage}
	\Prob\{L^\alpha \le p \le U^\alpha \ \text{on} \ \real\}
	\geq 1 - \alpha.
\end{equation}
In the context of calibration assessment, the functions $p, L^\alpha, U^\alpha$ are defined on $[0,1]$, and we call $(L^\alpha,U^\alpha)$ a calibration band, which is hence a confidence band for the calibration curve.
It allows for the desirable conclusion that with confidence $1-\alpha$, the true calibration curve $p$ lies inside the band, simultaneously for all values of the predicted probabilities. 
This nests a classical goodness-of-fit test with $\mathbb{H}_0 \colon p(x) = x$ by checking whether the band contains the diagonal $d(x) = x$ for all relevant values $x \in [0,1]$, but also any other hypothesis on the calibration curve such as e.g., an inverted goodness-of-fit test with $\mathbb{H}_0 \colon |p(x) - x| > \epsilon$ for some small $\epsilon>0$.
Hence, this band resolves the above mentioned criticisms of classical goodness-of-fit tests. 

Figure \ref{fig:ApplicationIntro} shows the bands in a large data example for probit model predictions for the binary outcome of a fetus having a low birth weight. 
See Section \ref{sec:ApplicationLBW} for additional details.
The test of Hosmer and Lemeshow clearly rejects calibration even though our bands indicate a well-calibrated model by including the diagonal line for all values in the unit interval. 
The magnified right panel of the figure shows that with confidence $1-\alpha$, the model is remarkably well calibrated for the most important region of small probability predictions in this application.
It is important to notice that even though we build our bands on the model predictions, the methodology applies equally to both, causal and predictive regressions.
An open-source implementation in the statistical software \texttt{R} \citep{R2022} is available under 
\href{https://cran.r-project.org/package=calibrationband}{https://cran.r-project.org/package=calibrationband}.

\begin{figure}[tb]
	\centering
	\includegraphics[width=\textwidth]{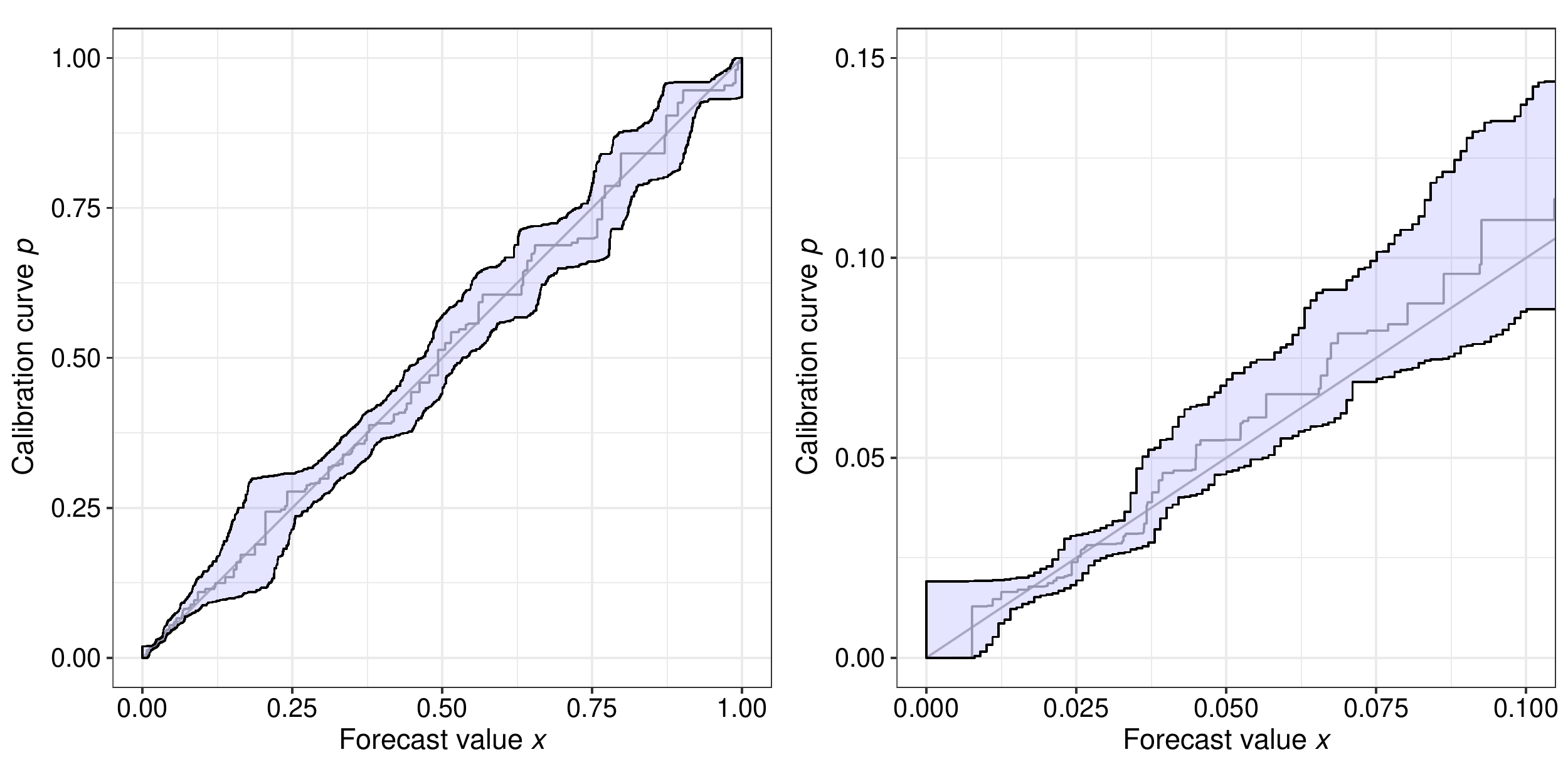}
	\caption{Left: Confidence band for the calibration curve (denoted calibration band in the application section) for the first model specification of the low birth weight application in Section \ref{sec:ApplicationLBW}. The blue band shows the confidence band based on the non-crossing method in \eqref{eq:NonCrossingBands} together with the rounding in \eqref{eqn:Rounding} with $K=10^3$, and the grey step function shows the isotonic regression estimate.
		Right: Magnified version focusing on predicted probabilities below $10\%$.}
	\label{fig:ApplicationIntro}
\end{figure}

Our confidence bands are valid in finite samples subject only to the mild monotonicity assumption at \eqref{eq:isotonicity}, implying that higher predictions entail a higher probability for $\{Y=1\}$, which is natural in the context of assessing calibration as already argued in \citet{DGJ_2021, Roelofs2020}. 
For classical goodness-of-fit tests, the null hypothesis $p(x) = x$ already nests the monotonicity assumption and if a researcher aims to demonstrate calibration, i.e., $p(x) = x$ holds at least approximately, it is unlikely that there are strong deviations from isotonicity.
Moreover, our confidence bands allow to detect and quantify violations of monotonicity as described in Appendix \ref{app:test_iso}. A non-monotonic $p$ may lead to a crossing of the lower and upper bound, i.e., $U^{\alpha}(x) < L^{\alpha}(x)$ for some $x$, which allows to reject monotonicity at level $\alpha$.
This is supported by the graphical display that reacts to non-isotonicity by generating elongated horizontal segments in both, the isotonic regression estimate and the confidence bands.
Finally, deriving confidence bands without any assumption on $p$ seems unrealistic and the assumption of monotonicity is relatively weak, e.g., in comparison to the parametric one used in \citet{Nattino2014}.

As expected for a non-parametric, pathwise and almost universally valid confidence band, we require large data sets of at least $5\,000$ observations to obtain sensibly narrow bands.
These are exactly the sample sizes where the classical goodness-of-fit tests become uninformative by rejecting all models in applications, see the simulation study of \citet{Kramer2007}.

A theoretical analysis shows that the proposed confidence band adapts locally to the smoothness of the function $p$ and to the variance of the observations. Adaptivity to the smoothness means that the width of the bands decreases faster with the sample size $n$ in regions where $p$ is constant, and at a slower rate where $p$ is steeper. This property is known for more general confidence bands for a monotone mean function developed by \citet{Yang2019}.
Adaptivity to the variance means that the band is substantially narrower at $x$ if $p(x)$ is close to zero or one, compared to $p(x)$ near $0.5$. In many practical applications, including the low birth weight predictions analyzed in this article, predicted probabilities close to zero or one are of most relevance and a sharp assessment of calibration in these regions is particularly important.

Existing methods for the construction of confidence bands in this setting are rare with the following two exceptions:
First, \cite{Nattino2014} propose the use of confidence bands based on a parametric assumption on the function $p$, which we show to have incorrect coverage in almost all of our simulation settings.
Second, the nonparametric bands of \citet{Yang2019} are valid, in a modified sense even in settings where the isotonicity assumption \eqref{eq:isotonicity} is violated. But they are shown to be wider than our bands in theory and simulations.

We explain the absence of competing methods by their theoretical difficulties.
Using asymptotic theory of the isotonic regression estimator is complicated as it requires the estimation of nuisance quantities such as the derivative of the unknown function $p$, the convergence rate depends on the functional form of $p$, it is subject to more restrictive assumptions and only results in bands with a pointwise interpretation \citep{Wright1981}.
Resampling schemes are theoretically found to be inconsistent for the isotonic regression \citep{Sen2010, GuntuboyinaSen2018}.
Other non-parametric approaches in the literature for constructing confidence bands for functions, many of them presented in the review by \citet{Hall2013}, are often pointwise, not simultaneous, and require the selection of tuning parameters that  may lead to instabilities, similar to the choice of the bins in the Hosmer-Lemeshow test. 
In contrast, the confidence bands proposed here are simple to compute and do not involve any implementation decisions resulting in a stable and reproducible method as called for by \citet{Stodden2016, Yu2020}.

\section{Construction of the confidence bands}
\label{sec:construction}

Within the regression setting, we construct confidence bands for the isotonic regression function $p$ by means of the classical confidence bounds of \cite{Clopper1934} for a binomial parameter.
Suppose that $Z$ is a binomial random variable with parameters $m$ and $q \in [0,1]$. For $\delta \in (0,1)$ let
\begin{align*}
	u^\delta(Z,m)
	&= \max\{\xi \in [0,1] \colon \mathrm{pbin}(Z,m,\xi) \ge \delta\} \\
	&= \begin{cases}
		\mathrm{qbeta}(1-\delta,Z+1,m-Z),
		& \ Z < m , \\
		1,
		& \ Z = m ,
	\end{cases} \\
	\ell^\delta(Z,m)
	&= \min\{\xi \in [0,1] \colon \mathrm{pbin}(Z-1,m,\xi) \le 1 - \delta\} \\
	&= \begin{cases}
		\mathrm{qbeta}(\delta,Z,m+1-Z),
		& \ Z > 0 , \\
		0,
		& \ Z = 0 .
	\end{cases}
\end{align*}
Here $\mathrm{pbin}(\cdot,m,\xi)$ denotes the distribution function of the binomial distribution with parameters $m$ and $\xi$, while $\mathrm{qbeta}(\cdot,a,b)$ stands for the quantile function of the beta distribution with parameters $a,b > 0$. Then
\[
\Prob\{q \le u^\delta(Z,m)\} \ge 1 - \delta
\quad\text{and}\quad
\Prob\{q \ge \ell^\delta(Z,m)\} \ge 1 - \delta .
\]
For the representation of $\ell^\delta(Z,m)$ and $u^\delta(Z,m)$ in terms of beta quantiles, we refer to \cite{Johnson2005}.

Assumption \eqref{eq:isotonicity} allows to construct confidence bands for $p$ as follows. With $p_i := p(x_i)$, for arbitrary indices $1 \le j \le k \le n$, the random sum
\[
Z_{jk} = \sum_{i=j}^k Y_i
\]
is stochastically larger than a binomial random variable with parameters $n_{jk} = k-j+1$ and $p_j$, and it is stochastically smaller than a binomial variable with parameters $n_{jk}$ and $p_k$. Thus, as explained in Lemma~\ref{lem:hoeffding},
\begin{equation}
	\label{eq:single.bounds}
	\Prob\{p_j \le u^\delta(Z_{jk},n_{jk})\} \ge 1 - \delta,
	\qquad
	\Prob\{p_k \ge \ell^\delta(Z_{jk},n_{jk})\} \ge 1 - \delta .
\end{equation}
If we combine these bounds for all pairs $(j,k)$ in a given set $\mathcal{J}$ and use the assumption at \eqref{eq:isotonicity}, then we may claim with confidence $1-2|\mathcal{J}| \delta$ that simultaneously for all $(j,k) \in \mathcal{J}$,
\[
p(x) \le u^\delta(Z_{jk},n_{jk}) \ \ \forall \ x \le x_j,
\qquad
p(x) \ge \ell^\delta(Z_{jk},n_{jk}) \ \ \forall \ x \ge x_k .
\]
Specifically, let $\mathcal{J}$ be the set of all index pairs $(j,k)$ such that $j \le k$ and $x_{j-1} < x_j$ and $x_k < x_{k+1}$.
If there are tied values in $(x_i)_{i=1}^n$, $\mathcal{J}$ selects the outermost indices of the tied values.
Hence, if $\{x_1,\ldots,x_n\}$ contains $N \le n$ different points, then $|\mathcal{J}| = (N^2 + N)/2$. Consequently, for a given confidence level $1 - \alpha \in (0,1)$, we may combine the bounds $u^\delta(Z_{jk},n_{jk})$ and $\ell^\delta(Z_{jk},n_{jk})$ with $\delta = \alpha/(N^2+N)$ to obtain a first confidence band.

\begin{theorem}
	\label{thm:Coverage}
	For $x \in \real$, let
	\begin{align}
		\label{eq:upper_bound}
		U^{\alpha, \mathrm{raw}}(x)
		&= \inf_{(j,k) \in \mathcal{J}\colon x_j \ge x}
		u^{\alpha/(N^2+N)}(Z_{jk},n_{jk}) , \\
		\label{eq:lower_bound}
		L^{\alpha, \mathrm{raw}}(x)
		&= \sup_{(j,k) \in \mathcal{J}\colon x_k \le x}
		\ell^{\alpha/(N^2+N)}(Z_{jk},n_{jk}) ,
	\end{align}
	where $\inf_\emptyset := 1$ and $\sup_\emptyset := 0$. If $p$ satisfies the isotonicity assumption at \eqref{eq:isotonicity}, then the resulting confidence band $(L^{\alpha,\mathrm{raw}},U^{\alpha,\mathrm{raw}})$ satisfies requirement \eqref{eq:coverage}.
\end{theorem}

The functions $U^{\alpha,\mathrm{raw}}, L^{\alpha,\mathrm{raw}}$ are isotonic and piecewise constant. Precisely, with $x_0 := -\infty$ and $x_{n+1} := \infty$, we know that $U^{\alpha,\mathrm{raw}} = 1$ on $(x_n,\infty)$, $L^{\alpha,\mathrm{raw}} = 0$ on $(-\infty,x_1)$, and
\begin{align*}
	U^{\alpha,\mathrm{raw}}(x) & = U^{\alpha,\mathrm{raw}}(x_i), \ \  x \in (x_{i-1},x_i] , \\
	L^{\alpha,\mathrm{raw}}(x) & = L^{\alpha,\mathrm{raw}}(x_i), \ \  x \in [x_i,x_{i+1}) ,
\end{align*}
for $i = 1, \dots, n$.
Consequently, computing the band $(L^{\alpha,\mathrm{raw}}, U^{\alpha,\mathrm{raw}})$ amounts to determining the $2n$ numbers $L^{\alpha,\mathrm{raw}}(x_i)$ and $U^{\alpha,\mathrm{raw}}(x_i)$, $i = 1, \dots, n$.

The confidence band proposed in Theorem~\ref{thm:Coverage} has two potential drawbacks. First, a natural nonparametric estimator for the function $p$ under the assumption \eqref{eq:isotonicity} is given by a minimizer $\hat{p}$ of $\sum_{i=1}^n \{h(x_i) - Y_i\}^2$ over all isotonic functions $h\colon [0,1] \to [0,1]$ \citep{DGJ_2021}. This minimizer is unique on the set $\{x_1,\ldots,x_n\}$. But there is no guarantee that $L^{\alpha,\mathrm{raw}} \le \hat{p} \le U^{\alpha,\mathrm{raw}}$. Second, the upper and lower bounds in \eqref{eq:upper_bound} and \eqref{eq:lower_bound} may even cross, resulting in an empty, and hence, nonsensical confidence band. 
These problems can be dealt with by using the non-crossing confidence band $(L^{\alpha,\mathrm{nc}}, U^{\alpha,\mathrm{nc}})$ given by pointwise minima and maxima: 
\begin{align}
	\label{eq:NonCrossingBands}
	L^{\alpha,\mathrm{nc}} = \min(L^{\alpha,\mathrm{raw}}, \hat{p}),
	\qquad
	U^{\alpha,\mathrm{nc}} = \max(L^{\alpha,\mathrm{raw}}, \hat{p}) .
\end{align}
Obviously, $L^{\alpha,\mathrm{nc}} \le \hat{p} \le U^{\alpha,\mathrm{nc}}$ on $\real$.
Our simulation experiments indicate that $(L^{\alpha,\mathrm{raw}}, U^{\alpha,\mathrm{raw}}) = (L^{\alpha,\mathrm{nc}},U^{\alpha,\mathrm{nc}})$ holds in almost all cases whenever $p$ satisfies \eqref{eq:isotonicity}; see Section \ref{sec:Simulations} for details. 
The potential crossing of the two bounds in Theorem \ref{thm:Coverage} also has an advantage. It allows for inference about the non-isotonicity of $p$, see Appendix \ref{app:test_iso}.

A potential obstacle in the practical application of the confidence bands proposed in this section is that their computation requires $\mathcal{O}(N^2)$ steps. This can be relieved by using a smaller family of index pairs $(j, k)$ in the definition of the confidence band. 
Specifically, if for some fixed integer $K \geq 1$ differences in the covariate smaller than $K^{-1}$ are regarded as negligible, then one could define
\begin{align}
	\label{eqn:Rounding}
	\widetilde{\mathcal{J}} = \{(j,k)\colon \{x_j\dots, x_k\} = \{x_1, \dots, x_n\} \cap [r/K, s/K] \text{ for some } r, s \in \mathbb{Z} \},
\end{align}
such that only blocks of covariate values between $r/K$ and $s/K$, $r,s \in \mathbb{Z}$, are considered. The resulting band is still honest, can be computed in $\mathcal{O}(|\widetilde{\mathcal{J}}|)$ steps, and one can reduce the correction factor of the significance level in \eqref{eq:upper_bound} and \eqref{eq:lower_bound} from $N^2+N$ to $|\widetilde{\mathcal{J}}|$. The drawback is that the constant regions in $L^{\alpha}$ and $U^{\alpha}$ become larger, thereby limiting the adaptivity of the band, so the number $K$ should not be too small.
We henceforth refer to the restricted choice of $\widetilde{\mathcal{J}}$ in \eqref{eqn:Rounding} as the rounding method.
Section 1 in the Supplementary Material illustrates in simulations that the rounding method drastically decreases the computation time and even results in narrower bands for all but very steep regions of $p$.

\section{Relation to \cite{Yang2019}}
\label{sec:Relation_to_YB}

The methods of \cite{Yang2019} may be adapted to the present regression setting with covariates $x_1 \le \cdots \le x_n$ as follows: With the isotonic estimator $\hat{p}$ introduced before, let
\[
Z_{jk}^{\mathrm{iso}} = \sum_{i=j}^k \hat{p}(x_i) .
\]
Set
\begin{align}
	\label{eq:upper_bound_YB}
	U^{\alpha, \mathrm{YB}}(x)
	& = \inf_{(j,k)\in \mathcal{J}\colon x_j \ge x}
	\Bigl[ \frac{Z_{jk}^{\mathrm{iso}}}{n_{jk}}
	+ \sqrt{\frac{\log\{(N^2 + N)/\alpha\}}{2 n_{jk}}} \Bigr] , \\
	\label{eq:lower_bound_YB}
	L^{\alpha, \mathrm{YB}}(x)
	& = \sup_{(j,k)\in \mathcal{J}\colon x_k \le x}
	\Bigl[ \frac{Z_{jk}^{\mathrm{iso}}}{n_{jk}}
	- \sqrt{\frac{\log\{(N^2 + N)/\alpha\}}{2 n_{jk}}} \Bigr] .
\end{align}
This defines a confidence band $(L^{\alpha,\mathrm{YB}},U^{\alpha,\mathrm{YB}})$ with the following property:
\begin{equation}
	\label{eq:coverage_YB}
	\Prob \{
	L^{\alpha,\mathrm{YB}} \le \tilde{p} \le U^{\alpha,\mathrm{YB}} 
	\ \ \text{on} \ \ \mathbb{R} \}
	\ge 1 - \alpha ,
\end{equation}
where $\tilde{p}\colon \mathbb{R} \to [0,1]$ is any fixed isotonic function minimizing $\sum_{i=1}^n \{\tilde{p}(x_i) - p_i\}^2$. Thus one obtains a confidence band with guaranteed coverage probability $1 - \alpha$ for an isotonic approximation of $p$, even if \eqref{eq:isotonicity} is violated. The proof of \eqref{eq:coverage_YB} follows from the arguments of \cite{Yang2019}, noting that the random variables $Y_i$ are sub-Gaussian with scale parameter $\sigma = 1/2$. Thus, $\mathbb{E} \exp(t(Y_i - p_i)) \le \exp(\sigma^2t^2/2)$ for all $t \in \real$, implying that for arbitrary $\eta \ge 0$,
\[
\Prob\{ \pm (Z_{jk} - \mathbb{E} Z_{jk}) \ge \eta\}
\le \exp(-2 n_{jk} \eta^2) ,
\]
see \cite{Hoeffding1963}. The following result shows that the confidence bands $(L^{\alpha,\mathrm{raw}},U^{\alpha,\mathrm{raw}})$ and $(L^{\alpha,\mathrm{nc}},U^{\alpha,\mathrm{nc}})$ are always contained in the band $(L^{\alpha,\mathrm{YB}},U^{\alpha,\mathrm{YB}})$.

\begin{theorem}
	\label{thm:iso}
	For all $\alpha \in (0, 1)$ and any data vector $\mathcal{Y} \in \{0,1\}^n$,
	\[
	L^{\alpha,\mathrm{YB}} \le L^{\alpha,\mathrm{nc}} \le L^{\alpha,\mathrm{raw}},
	\quad
	U^{\alpha,\mathrm{raw}} \le U^{\alpha,\mathrm{nc}} \le U^{\alpha,\mathrm{YB}}
	\ \ \text{on} \ \ \mathbb{R}.
	\]
\end{theorem}

Recall that the inequalities $L^{\alpha,\mathrm{raw}} \leq	U^{\alpha,\mathrm{raw}}$ do not hold in general, and a crossing of the bounds allows to reject isotonicity at level $\alpha$, see Appendix \ref{app:test_iso}. In contrast, the bands by \citet{Yang2019} always contain the isotonic estimator $\hat{p}$, and are guaranteed to cover an optimal isotonic approximation of $p$ with probability at least $1-\alpha$. For calibration testing, the possibility of rejecting isotonicity seems more desirable than information about an isotonic approximation of the calibration curve, whose interpretation may be unclear in practice. 
It should be mentioned, however, that the band $(L^{\alpha,\mathrm{YB}},U^{\alpha,\mathrm{YB}})$ has a computational advantage. For the computation of $U_i^{\alpha,\mathrm{YB}}$ in \eqref{eq:upper_bound_YB}, it suffices to take the minimum over endpoints of constancy regions of $\hat{p}$, that is, all $(j,k) \in \mathcal{J}$ such that $j = \min(s\colon x_s\ge x_i)$ and $\hat{p}(x_k) < \hat{p}(x_{k+1})$ or $k = n$, see Proposition~\ref{prop:isobounds}. Likewise, for the computation of $L_i^{\alpha,\mathrm{YB}}$ in \eqref{eq:lower_bound_YB}, it suffices to take the maximum over all $(j,k) \in \mathcal{J}$ such that $\hat{p}(x_{j-1}) < \hat{p}(x_j)$ or $j = 1$ and $k = \max(s\colon x_s \le x_i)$. While the computation of $(L^{\alpha,\mathrm{raw}},U^{\alpha,\mathrm{raw}})$ or $(L^{\alpha,\mathrm{nc}},U^{\alpha,\mathrm{nc}})$ requires $\mathcal{O}(N^2)$ steps, the following lemma, whose proof is in the Supplementary Material, implies that the computation of $(L^{\alpha,\mathrm{YB}},U^{\alpha,\mathrm{YB}})$ requires only $\mathcal{O}(N \min\{n^{2/3},N\})$ steps.

\begin{lemma} 
	\label{lem:isopart}
	The cardinality of $\{\hat{p}(x_i)\colon i = 1, \dots, n\}$ is smaller than $3 n^{2/3}$.
\end{lemma}

\section{Theoretical properties of the confidence bands}
\label{sec:iso}

This section illustrates consistency and adaptivity properties of the confidence band $(L_n^{\alpha,\mathrm{raw}},U_n^{\alpha,\mathrm{raw}})$, where the subscript $n$ indicates the sample size, and we consider a triangular scheme of observations $(x_i,Y_i) = (x_{ni},Y_{ni})$, $i = 1, \dots, n$. We are interested in situations in which the observed covariates $x_{ni}$ could be the realizations of the order statistics of a random sample. Thus we extend the framework of \cite{Yang2019} and consider the following assumption.

\begin{assumption}
	\label{ass:A}
	Let $\mathrm{Leb}(\cdot)$ denote Lebesgue measure, and let $W_n(B) = \#\{i\colon x_{ni} \in B\}$ for $B \subset \real$. There exist a non-degenerate interval $[a_o,b_o]$ and constants $C_1, C_2 > 0$ such that for sufficiently large $n$, 
	\[W_n(B) \ge C_1 n \mathrm{Leb}(B)\]
	for arbitrary intervals $B \subset [a_o,b_o]$ such that $\mathrm{Leb}(B) \ge C_2 \log(n)/n$.
\end{assumption}

This assumption comprises the setting of \cite{Yang2019}. Let $G$ be a differentiable distribution function on $[0,1]$ such that $G'$ is bounded away from $0$ on $[a_o,b_o]$. If $x_{ni} = G^{-1}(i/n)$ for $i = 1, \dots, n$, then it is satisfied for any $C_1 < \inf_{[a_o,b_o]} G'$ and arbitrary $C_2 > 0$. The arguments in \citet[Section~4.3]{Moesching2020} can be modified to show that if $x_{n1},\ldots,x_{nn}$ are the order statistics of $n$ independent random variables with distribution function $G$, then Assumption \ref{ass:A} is satisfied almost surely, provided that $C_1, C_2 > 0$ are chosen appropriately.

\begin{theorem}
	\label{thm:asymptotics}
	Suppose that Assumption \ref{ass:A} is satisfied. Let $\rho_n = \log(n)/n$. There exist constants $C > 0$ depending only on $C_1,C_2$ with the following properties:
	
	\noindent
	(i) Suppose that $p$ is constant on $[a_o,b_o]$. With asymptotic probability one,
	\begin{align*}
		U_n^{\alpha,\mathrm{raw}}(x) &\le p(x) + C \sqrt{\rho_n/(b_o - x)},
		\quad x \in [a_o,b_o) , \\
		L_n^{\alpha,\mathrm{raw}}(x) &\ge p(x) - C \sqrt{\rho_n/(x - a_o)},
		\quad x \in (a_o,b_o] .
	\end{align*}
	
	\noindent
	(ii) Suppose that $p$ is Lipschitz-continuous on $[a_o,b_o]$ with Lipschitz constant $L > 0$. With asymptotic probability one,
	\begin{align*}
		U_n^{\alpha,\mathrm{raw}}(x) &\le p(x) + C (L\rho_n)^{1/3} ,
		\quad x \in [a_o, b_o - \rho_n^{1/3} L^{-2/3}] , \\
		L_n^{\alpha,\mathrm{raw}}(x) &\ge p(x) - C (L \rho_n)^{1/3} ,
		\quad x \in [a_o + \rho_n^{1/3} L^{-2/3}, b_o] .
	\end{align*}
	
	\noindent
	(iii) Suppose that $p$ is discontinuous at some point $x_o \in (a_o,b_o)$. With asymptotic probability one,
	\begin{align*}
		U_n^{\alpha,\mathrm{raw}}(x) &\le p(x_o-) + C \sqrt{\rho_n /(x_o - x)} ,
		\quad x \in [a_o,x_o) , \\
		L_n^{\alpha,\mathrm{raw}}(x) &\ge p(x_o+) - C \sqrt{\rho_n /(x - x_o)} ,
		\quad x \in (x_o,b_o] .
	\end{align*}
	
	\noindent
	(iv) Suppose that $\lim_{x \to a_o} p(x) = 0$. For sufficiently large $n$,
	\[
	\mathbb{E} \{U_n^{\alpha,\mathrm{raw}}(x)\}
	\le C \inf_{y \in (x,b_o]} \{ p(y) + \rho_n/(y - x) \} ,
	\quad x \in [a_o, b_o) .
	\]
	Analogously, if $\lim_{x \to b_o} p(x) = 1$, then for sufficiently large $n$,
	\[
	\mathbb{E} \{1 - L_n^{\alpha,\mathrm{raw}}(x)\}
	\le C \inf_{y \in [a_o,x)} \{ 1 - p(y) + \rho_n/(x - y) \} ,
	\quad x \in (a_o,b_o] .
	\]
\end{theorem}

Part~(i) implies that if $p$ is constant on $[a_o,b_o]$, then for arbitrary fixed $a_o < a < b < b_o$,
\[
\sup_{x \in [a_o,b]} \{U_n^{\alpha,\mathrm{raw}}(x) - p(x)\}^+ +
\sup_{x \in [a_o,b]} \{p(x) - L_n^{\alpha,\mathrm{raw}}(x)\}^+
= \mathcal{O}_p(\rho_n^{1/2}).
\]
Thus, parts~(i-ii) of this theorem are analogous to results of \citet[Sections 4.4 and 4.6]{Yang2019}.
Part~(iii) implies that with asymptotic probability one,
\[
U_n^{\alpha,\mathrm{raw}}(x)
< \frac{p(x_o-) + p(x_o+)}{2}
< L_n^{\alpha,\mathrm{raw}}(y)
\]
for $x < x_o - D \rho_n$, $y > x_o + D \rho_n$ and $D = 4 C^2 \{p(x_o+) - p(x_o-)\}^{-2}$. Thus, at points of discontinuity of $p$, the confidence band crosses a horizontal line on an interval of length $\mathcal{O}_p(\rho_n)$. Part~(iv) demonstrates that our bounds are particularly accurate in regions where $p(x)$ is close to $0$ or $1$. Specifically, suppose that for some $\gamma > 0$, $p(x) = \mathcal{O}\{(x - a_o)^\gamma\}$ for $x \in [a_o,b_o]$. Then plugging in $y(x) = x + \rho_n^{1/(\gamma + 1)}$ reveals that
\[
\mathbb{E}\{U_n^{\alpha,\mathrm{raw}}(x)\}
\le D \{(x - a_o)^\gamma + \rho_n^{\gamma/(\gamma + 1)}\} ,
\quad x \in [a_o,b_o] ,
\]
where $D = D(C_1,C_2,p)$. Analogously, if $1 - p(x) = \mathcal{O}\{(b_o - x)^\gamma\}$ for $x \in [a_o,b_o]$, then
\[
\mathbb{E}\{1 - L_n^{\alpha,\mathrm{raw}}(x)\}
\le D \{(b_o - x)^\gamma + \rho_n^{\gamma/(\gamma + 1)}\} ,
\quad x \in [a_o,b_o] .
\]
Presumably, the conclusions in part~(iv) are not satisfied for the confidence band $(L_n^{\alpha,\mathrm{YB}},U_n^{\alpha,\mathrm{YB}})$.

\section{Simulations}
\label{sec:Simulations}

Here, we illustrate that our confidence bands have correct coverage in the sense of \eqref{eq:coverage} and are narrower than existing techniques.
We consider both, the raw method in \eqref{eq:upper_bound} and \eqref{eq:lower_bound} and the non-crossing variant in \eqref{eq:NonCrossingBands}.
Both methods are combined with the rounding technique in \eqref{eqn:Rounding} with $K=10^3$ in order to facilitate faster computation at a minimal cost in accuracy.
For comparison, we use the bands of \cite{Yang2019} given in \eqref{eq:upper_bound_YB} and  \eqref{eq:lower_bound_YB} with a minimal variance factor of $\sigma^2 = 1/4$ and the parametric bands of \cite{Nattino2014}, implemented in the \texttt{GivitiR} package in the statistical software \texttt{R} \citep{R2022}.
Replication material for the simulations and applications is available under  \href{https://github.com/marius-cp/replication_DDHPZ22}{https://github.com/marius-cp/replication\_DDHPZ22}.

\begin{figure}[tb]
	\centering
	\includegraphics[width=\textwidth]{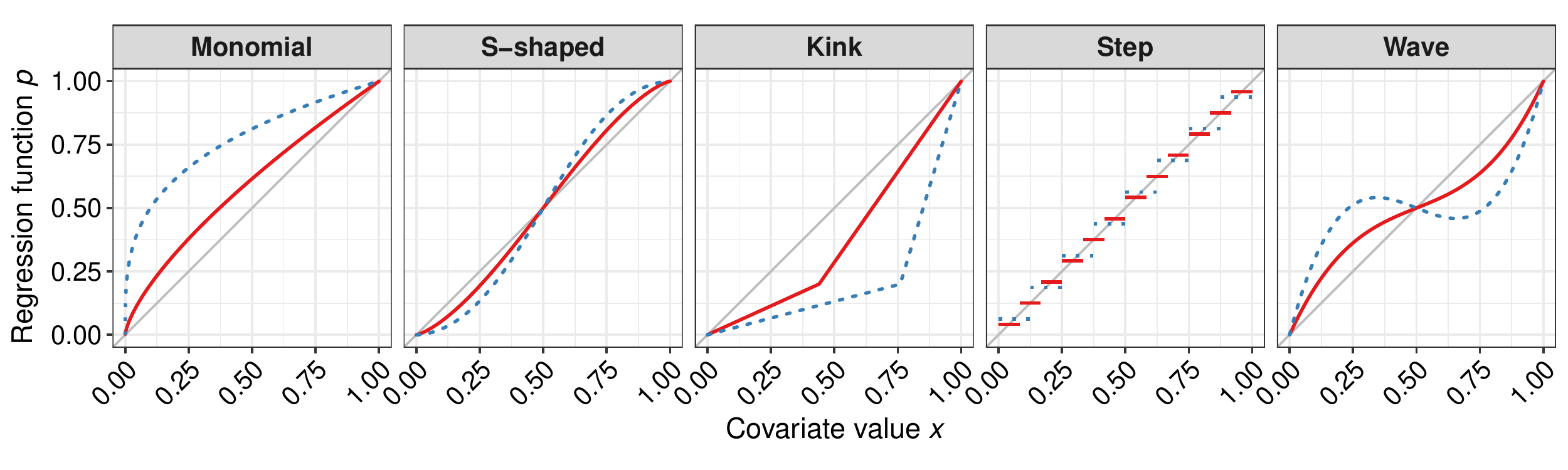}
	\caption{Illustration of the five simulated regression functions $p_s(\cdot)$, where the solid red line corresponds to the shape parameter value $s=0.3$ and the dashed blue line to $s=0.7$.}
	\label{fig:p_alternatives}
\end{figure} 

We use 1000 replications, a significance level of  $\alpha = 0.05$ and simulate the covariates $X \sim \operatorname{U}[0,1]$.
The binary outcomes are generated by $Y \sim \text{Bern}\{p_s(X)\}$ based on five distinct functional forms of the regression function $p_s(x)$ for $x \in [0,1]$ depending on a shape parameter $s \in \mathcal{S} := \{0, 0.1, \dots, 1\}$.
The first four specifications of $p_s(x)$ satisfy the isotonicity assumption at \eqref{eq:isotonicity} and cover smooth, non-smooth as well as discontinuous setups.
The last one contains non-isotonic functions $p_s(x)$ for $s > 0.5$.
The choice $s=0$ results in the diagonal line $p_0(x) = x$ whereas the deviation from the diagonal increases with $s$.
In particular, we consider the following specifications, which are illustrated in  Figure \ref{fig:p_alternatives} for two exemplary shape values $s \in \{0.3, 0.7\}$.

\begin{enumerate}
	\item  
	Monomial: \
	First, we use the regression function $p_s(x) = x^{1-s}$, where  $s \in \mathcal{S} \setminus \{1\}$.
	This function is already used in the simulations in \citet[Appendix A]{DGJ_2021}.
	
	\item
	S-shaped: \
	Second, the regression function follows an S-shaped form $p_s(x) = \left( 1+((1-x)/x)^{1+s} \right)^{-1}$, where $s \in \mathcal{S}$ pronounces the curves for larger values of $s$.
	
	\item   
	Kink: \ 
	Third, $p_s(x)$ linearly interpolates the points $(0,0), (0.2 + 0.8s, 0.2)$ and $(1,1)$ for $s \in \mathcal{S}$, resulting in a kink at the point $(0.2 + 0.8s, 0.2)$ for all $s > 0$.
	
	\item  
	Step: \  
	Fourth, we use a step function with $s^\star \in \{5,6,\dots,14\}$ equidistant steps in the unit interval.
	It is given by $p_s(x) =  \big\{ \lfloor s^\star x  \rfloor + \mathds{1}(x \not= 1) \big\} / s^\star$, where $s^\star = 15 - 10s$ and $s \in \mathcal{S} \setminus \{0\}$.
	It doesn't nest the diagonal, but the deviation from it increases with $s$.
	
	\item  
	Wave: \
	Fifth, we use the cubic function $p_s(x) = 0.5 - (2s-1)(x-0.5) + 8s (x-0.5)^3$ that violates the isotonicity assumption in  \eqref{eq:isotonicity} for any $s > 0.5$. 
\end{enumerate}

Figure \ref{fig:coverage_unif} presents the average coverage rates for a range of sample sizes between 512 and 32\,768.
We use the raw method for our bands in \eqref{eq:upper_bound} and \eqref{eq:lower_bound} as the raw bands are contained in the non-crossing ones.
We find that, as predicted by the theory, our confidence bands have conservative coverage throughout all isotonic simulation setups and sample sizes.
We observe coverage rates above 0.998 with the majority of 162 out of the 192 displayed coverage values being exactly one.
The unreported non-crossing bands differ from the raw ones in less than one out of a hundred thousand instances.
These deviations occur exclusively for large values of $s$ in the Step specification within constancy regions of the function $p$.
As expected, our method as well as the bands of \cite{Yang2019} have incorrect coverage rates for the values $s > 0.5$ that violate isotonicity in the Wave specification when the sample size increases.
The coverage rates of the \cite{Yang2019} bands are still larger as these are shown to be wider by Theorem \ref{thm:iso}. 	

\begin{figure}[tb]
	\centering
	\includegraphics[width=\textwidth]{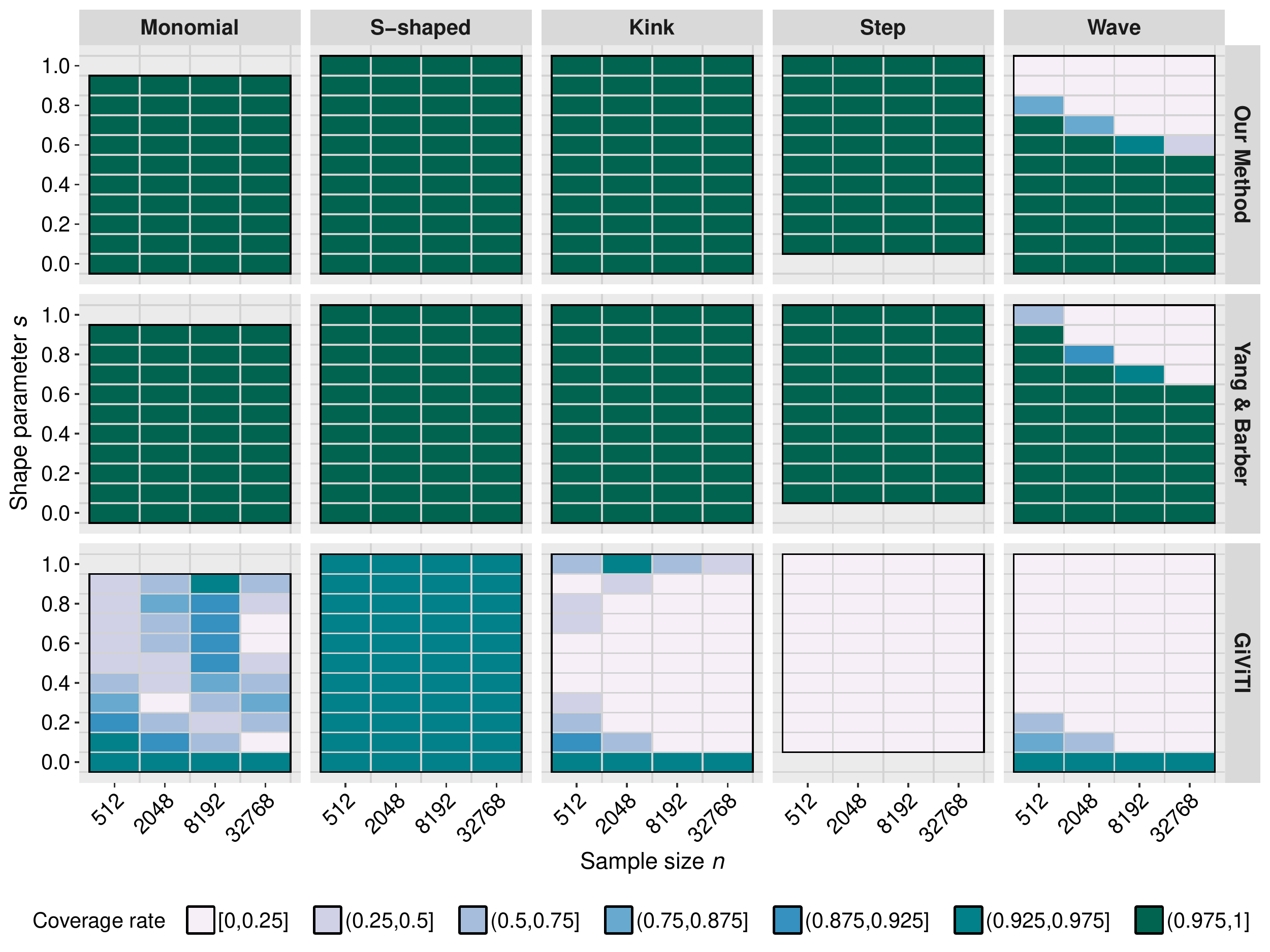}
	\caption{Empirical coverage rates of our confidence bands, the bands of \citet{Yang2019}, and the GiVitI bands for $1-\alpha = 0.95$, averaged over all covariate values  for the five specifications of the regression function $p_s(\cdot)$, different shape values $s$ and a range of sample sizes $n$.
		For our bands, we use the raw method in \eqref{eq:upper_bound} and \eqref{eq:lower_bound}, together with rounding in \eqref{eqn:Rounding} with $K=10^3$.
		The choices $s=1$ in the Monomial, and $s=0$ in Step specification are not defined.}
	\label{fig:coverage_unif}
\end{figure}

The parametric bands of \cite{Nattino2014} rarely achieve correct coverage rates unless in the cases $s=0$ and for the S-shaped regression functions.
This can be explained as these bands are based on the assumption of a certain parametric form of $p_s(x)$, which is rarely satisfied. 
The results get worse for the non-smooth, the discontinuous and the non-isotonic specifications.

\begin{figure}[tb]
	\centering
	\includegraphics[width=\textwidth]{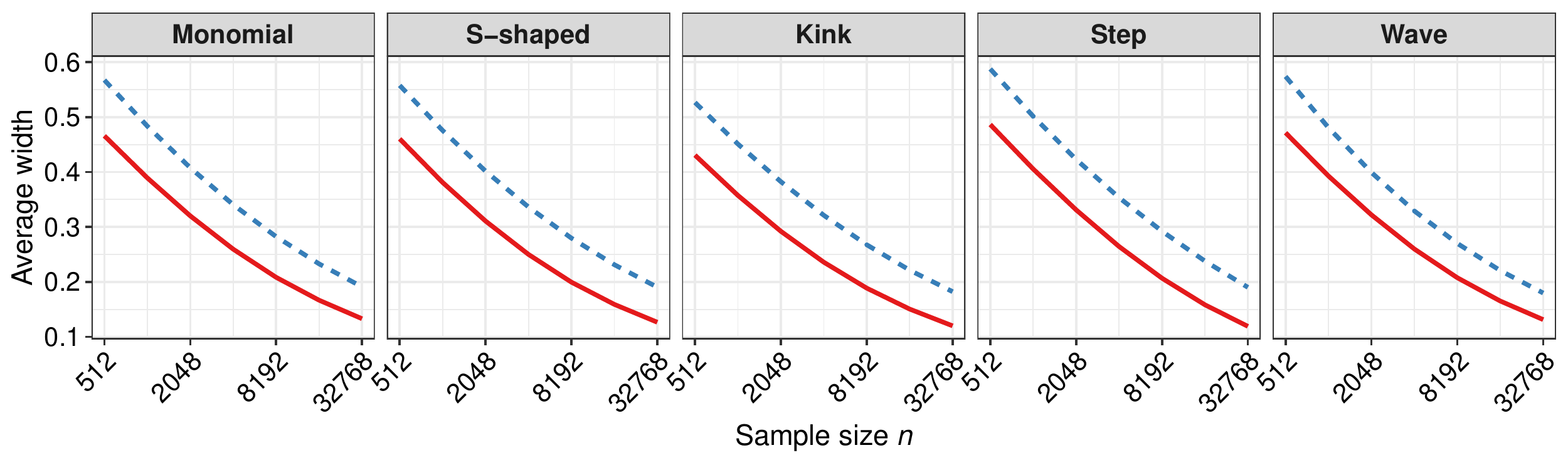}
	\includegraphics[width=\textwidth]{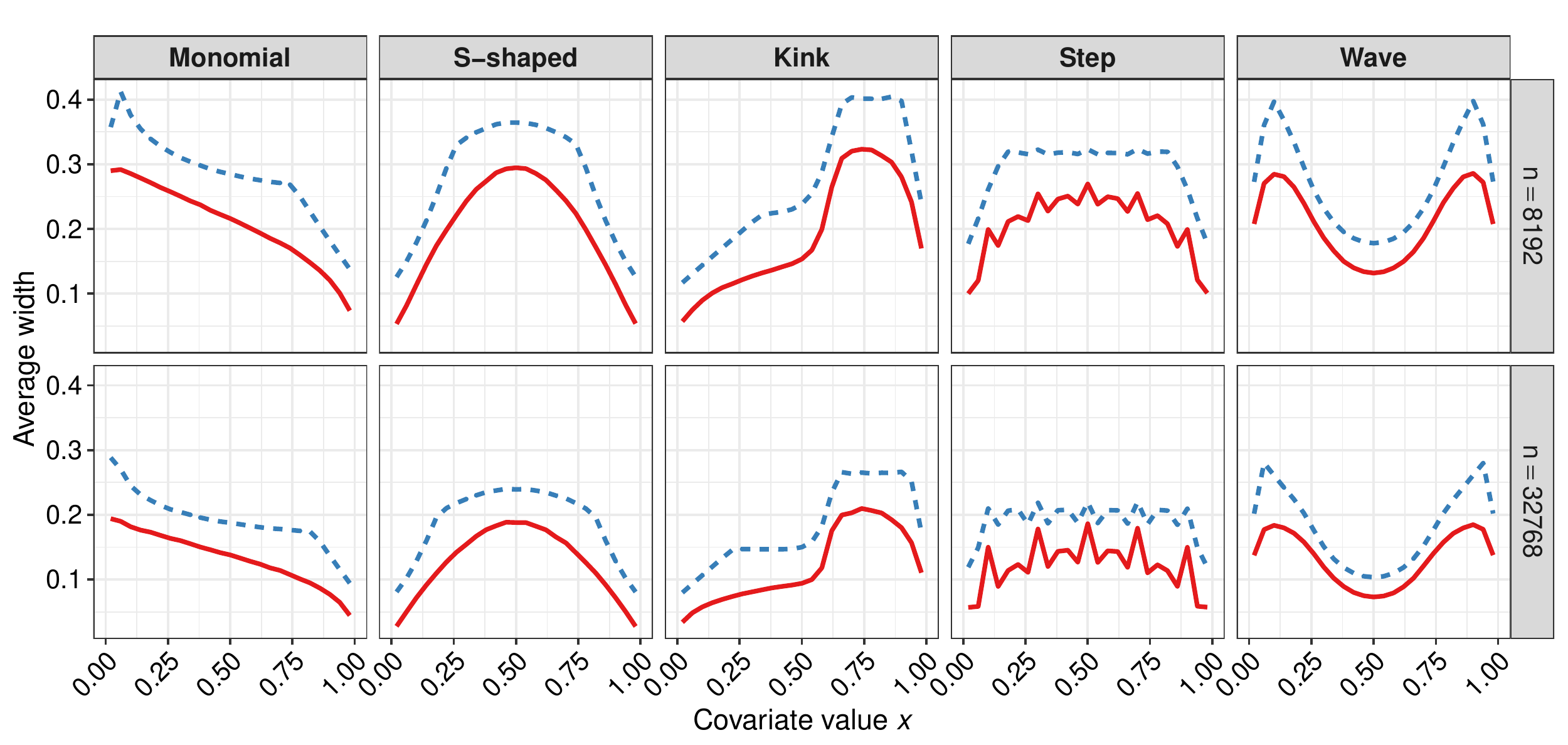}
	\caption{Top: Average widths of the $95\%$ confidence bands by sample size for each of the five specifications of $p_s(x)$ given in the main text for a fixed value $s=0.5$.
		Bottom: Average widths by covariate value $x$ for two sample sizes.
		In both panels, the solid red line corresponds to our bands based on the non-crossing method in \eqref{eq:NonCrossingBands} together with rounding in \eqref{eqn:Rounding} with $K=10^3$, and the dashed blue line corresponds to the \cite{Yang2019} bands.}
	\label{fig:width_unif}
\end{figure}

Figure \ref{fig:width_unif} displays the average widths of our and the \cite{Yang2019} bands.
We present the theoretically wider non-crossing bands instead of the raw versions thereof.
Their average widths is however non-distinguishable in these displays.
We fix a medium degree of miscalibration $s=0.5$.
The upper plot panel displays the widths averaged over all simulation runs and values $x \in [0,1]$ depending on the sample size $n$.
We find that the size of both bands shrinks with $n$ and that we can reconfirm the ordering established in Theorem \ref{thm:iso}.
We further see that our bands are only narrow enough for practical use in large samples.
The relative gain in width of our bands is the highest for large sample sizes, exactly for which we propose the application of our method for calibration validation.
It is worth noting that the bands of \cite{Yang2019} are more generally valid than for the special case of binary observations.

The lower plot panel shows the widths averaged over the simulation replications, but depending on the values $x \in [0,1]$ for two selected sample sizes.
It shows that the relative gains in width upon the bands of \cite{Yang2019} are particularly pronounced close to the edges of the unit interval.
In applications to calibration assessment, these regions of predicted probabilities close to zero or one are often of the highest interest as for example in the subsequent section assessing the goodness-of-fit of low birth weight probability predictions.

\section{Application: Predicting low birth weight probabilities}
\label{sec:ApplicationLBW}

We apply our confidence bands to assess calibration of three binary regression specifications predicting the probability of a fetus having a low birth weight, defined as weighting less than 2500 grams at birth \citep{def_who}.
Recall that in the setting of calibration assessment, we call the function $p$ the calibration curve and our confidence bands are denoted as calibration bands. 
This follows the interpretation that for an event predicted with probability $x$, $p(x)$ denotes its true but unknown event probability.
Perfectly calibrated predictions entail a calibration curve matching the diagonal line, $d(x) = x$.
As the calibration band is a simultaneously valid confidence band for $p$, deviations of the calibration band from the diagonal line imply significantly miscalibrated predictions in this region.

\begin{figure}[tb]
	\centering
	\includegraphics[width=\textwidth]{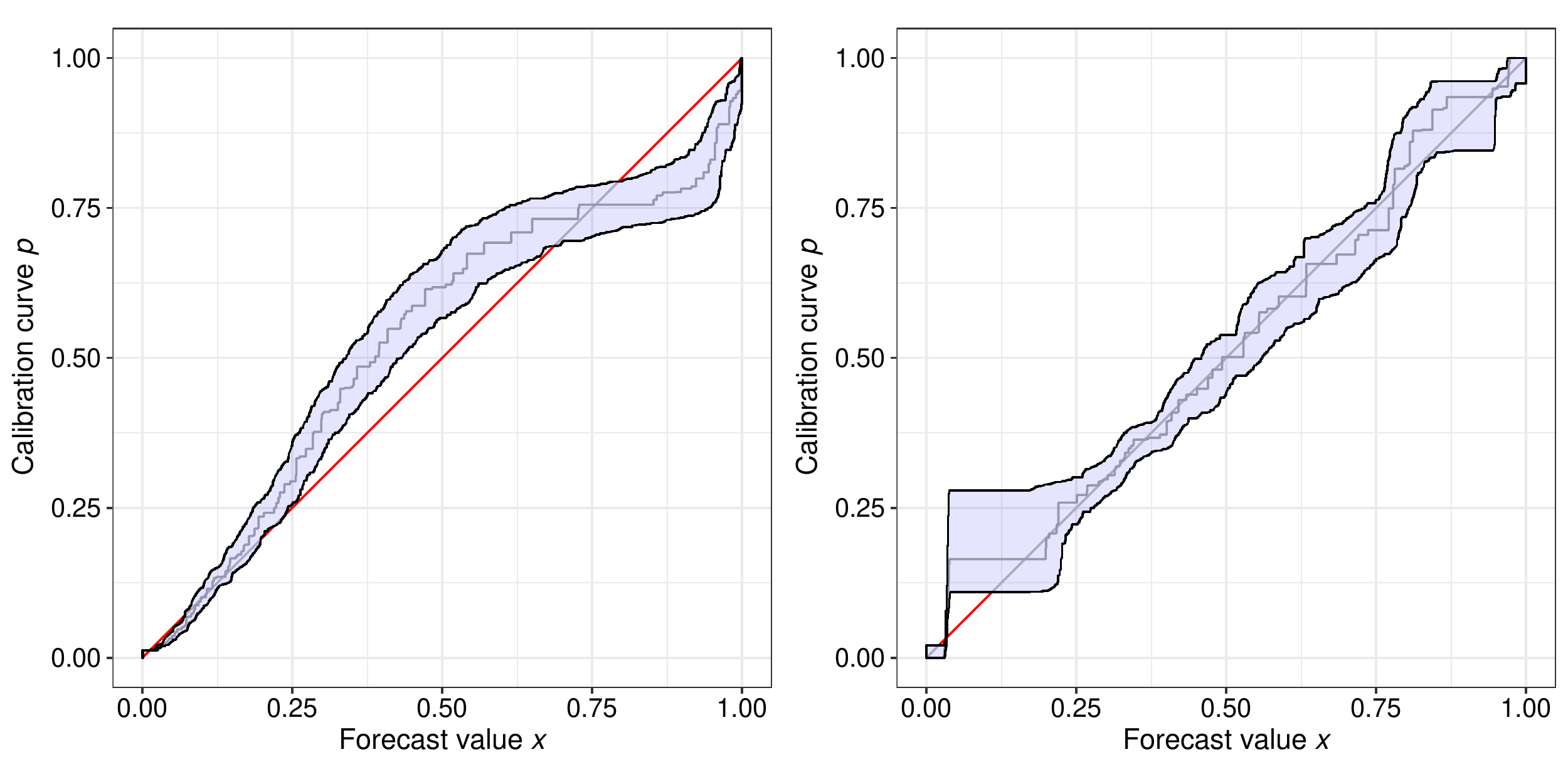}
	\caption{Calibration bands for the second model specification on the left and for the third specification on the right for the low birth weight application.
		The blue band denotes the calibration band based on the non-crossing method in \eqref{eq:NonCrossingBands} together with the rounding in \eqref{eqn:Rounding} with $K=10^3$, and the grey step function shows the isotonic regression estimate. The diagonal line is given in red color whenever it is not contained in the calibration band.}
	\label{fig:appl_lbw}
\end{figure}

We use U.S.~Natality Data from the \citet{Natality2017}, which provides demographic and health data for 3\,864\,754 births in the year 2017.
For the data set at hand, a low birth weight is observed in 8.1\% of the cases. 
We estimate three binary regression models by maximum likelihood on the same randomly drawn subset that contains all but 1\,000\,000 observations that we leave for external model validation.
All three models contain standard risk factors such as the mother's age, body mass index and smoking behavior but they differ as follows.
The first model uses a probit link function, and the explanatory variable week of gestation is categorized into four left-closed and right-open intervals with lower interval limits of 0, 28, 32 and 37 weeks, pertaining to the standard definitions of the World Health Organization of extremely, very, moderate and non preterm \citep{Quinn2016PretermDef}.
Through this categorization, the model specification can capture the week of gestation in a non-linear fashion. 
In contrast, the second model uses the week of gestation as a continuous explanatory variable and the third specification employs the cauchit instead of the probit link function, which is known to produce less confident predictions close to zero and one \citep{Koenker2009}.
Additional details of the model specifications are given in the Supplementary Material.

The classical Hosmer-Lemeshow test rejects perfect calibration of all three models with p-values of essentially zero for both, internal and external model validation, which leaves an applied researcher without any useful conclusions on model calibration.
We show our calibration bands based on the non-crossing method with rounding to three digits, i.e., $K=10^3$ in \eqref{eqn:Rounding}, with a confidence level of $1-\alpha = 95\%$ for the first model in Figure \ref{fig:ApplicationIntro} and for the other two models in Figure \ref{fig:appl_lbw}.
We constantly extrapolate the bands on the unit interval which preserves their theoretical coverage guarantees as discussed after Theorem \ref{thm:Coverage}.
Figure S3 in the Supplementary Material illustrates that the bands of \citet{Yang2019} are considerably wider in this application.

Recall that the validity of our bands relies on the isotonicity assumption of $p$, which we test for as detailed on in Appendix \ref{app:test_iso}.
The test only rejects isotonicity at the $5\%$ level  for the second model specification displayed on the left side of Figure \ref{fig:appl_lbw} with a crossing of the lower and upper bounds for probability predictions between $0.1\%$ and $2.7\%$.
Hence, we can directly reject calibration for this model in the critical area of small predictions and furthermore, the remaining calibration band has to be interpreted carefully for this model.
As the simulations in Appendix \ref{app:test_iso} show that the isotonicity test can even detect slight violations of isotonicity with high power for much smaller sizes as considered in this application, type II test errors are barely a problem here and we can be  confident about the isotonicity assumption for the other two model specifications.

For the first model, the calibration band encompasses the diagonal line for all forecast values, meaning that we cannot reject the null hypothesis of perfect calibration $p(x) = x$ at the $5\%$ level.
More importantly, we are $95\%$ certain that the true calibration curve lies within the the band at any point $x \in [0,1]$, implying that we are confident that the model is at least as well calibrated as specified by the band.
This is especially notable in the important region of predictions below $10\%$ in the magnified right panel of  Figure \ref{fig:ApplicationIntro}, where the calibration bands are remarkably close to the diagonal implying a particularly well calibrated model.
E.g., we can conclude that for a prediction of $x=5\%$, a low birth weight occurs with a probability between $4.6\%$ and $6.7\%$.

In contrast, we reject calibration for both, the second and third model specifications as shown in Figure \ref{fig:appl_lbw}.
However, these bands are much more informative than a simple test rejection as they directly show the exact form of model miscalibration.
For the second model specification, we can conclude that the predicted probabilities are particularly miscalibrated for the non-isotonic region discussed above and for values larger than $20\%$.
The third specification entails miscalibrated probabilities for predictions below $10\%$ that are presumably of the highest importance for medical decision making.
Finally notice that the wide bands for the third model specification between predicted probabilities of $5\%$ and $20\%$ are caused by little predictions in this interval.

\section*{Acknowledgement}
T.~Dimitriadis gratefully acknowledges financial support from the German Research Foundation (DFG) through grant number 502572912.
A.~Henzi and J.~Ziegel gratefully acknowledge financial support from the Swiss National Science Foundation.

\section*{Supplementary material}
The Supplementary Material further illustrates the rounding method in simulations, gives details on the low birth weight application and contains additional proofs.

\appendix

\section{Detecting and quantifying non-isotonicity} 
\label{app:test_iso}

\renewcommand{\thetheorem}{A\arabic{theorem}}   
\renewcommand{\theequation}{A\arabic{equation}}   

\setcounter{theorem}{0}
\setcounter{equation}{0}

The regression function $p$ could violate isotonicity in \eqref{eq:isotonicity}. Then, its non-isotonicity can be quantified by 
\[
\gamma(p) := \sup_{x \le y} \{p(x) - p(y)\} \ge 0 .
\]
The derivation of our confidence band $(L^{\alpha,\mathrm{raw}}, U^{\alpha,\mathrm{raw}})$ can be adapted as follows: For any index pair $(j,k) \in \mathcal{J}$ and $\delta \in (0,1)$, we know that
\begin{equation}
	\label{eq:single.bounds.v2}
	\mathbb{P}[\min\{p_j,\ldots,p_k\} \le u^\delta(Z_{jk},n_{jk})] \ge 1 - \delta,
	\quad
	\mathbb{P}[\max\{p_j,\ldots,p_k\} \ge \ell^\delta(Z_{jk},n_{jk})] \ge 1 - \delta .
\end{equation}
But the definition of $\gamma(p)$ implies that
\[
p(x) \le \min\{p_j,\ldots,p_k\} + \gamma(p) \ \ \forall \ x \le x_j,
\qquad
p(x) \ge \max\{p_j,\ldots,p_k\} - \gamma(p) \ \ \forall \ x \ge x_k .
\]
Consequently, one can complement Theorem~\ref{thm:Coverage} with the following result:

\begin{theorem}
	\label{thm:CoverageExtended}
	Let $(L^{\alpha, \mathrm{raw}},U^{\alpha, \mathrm{raw}})$ be defined as in Theorem~\ref{thm:Coverage} . Then for any regression function $p$,
	\[
	\mathbb{P}\{L^{\alpha,\mathrm{raw}} - \gamma(p)
	\le p \le U^{\alpha,\mathrm{raw}} + \gamma(p) \}
	\ge 1 - \alpha .
	\]
\end{theorem}

This result has two implications: First, a p-value for the null hypothesis that $p$ is isotonic is given by the supremum of all $\alpha \in (0,1)$ such that $L^{\alpha,\mathrm{raw}} \le U^{\alpha,\mathrm{raw}}$ pointwise. 
Second, for a fixed $\alpha \in (0,1)$ let $\hat{\gamma}_{\alpha} \ge 0$ be the infimum of all numbers $\gamma \ge 0$ such that $L^{\alpha,\mathrm{raw}} - \gamma \le U^{\alpha,\mathrm{raw}} + \gamma$. In other words, $\hat{\gamma}_\alpha$ equals $\sup_{x \in \mathbb{R}} \{L^{\alpha,\mathrm{raw}}(x) - U^{\alpha,\mathrm{raw}}(x)\}^+ / 2$. Then $\hat{\gamma}_\alpha$ is a lower $(1 - \alpha)$-confidence bound for $\gamma(p)$.

\begin{table}
		\centering
		\caption{Rejection rates of the isotonicity test}
		\begin{tabular}{rrrrrrrr}
			\toprule
			\multicolumn{1}{c}{ } & \multicolumn{7}{c}{Sample size $n$} \\
			\cmidrule(l{3pt}r{3pt}){2-8}
			$s$ & 512 & 1024 & 2048 & 4096 & 8192 & 16384 & 32768\\
			\midrule
			0.5 & 0.00 & 0.00 & 0.00 & 0.00 & 0.00 & 0.00 & 0.00\\
			0.6 & 0.00 & 0.00 & 0.00 & 0.00 & 0.00 & 0.00 & 0.00\\
			0.7 & 0.00 & 0.00 & 0.00 & 0.00 & 0.01 & 0.16 & 0.89\\
			0.8 & 0.00 & 0.00 & 0.01 & 0.08 & 0.71 & 1.00 & 1.00\\
			0.9 & 0.00 & 0.01 & 0.13 & 0.85 & 1.00 & 1.00 & 1.00\\
			1.0 & 0.01 & 0.14 & 0.81 & 1.00 & 1.00 & 1.00 & 1.00\\
			\bottomrule
	\end{tabular}
	\label{tab:IsotonicityTest}
\end{table}

Table \ref{tab:IsotonicityTest} illustrates the isotonicity test's performance using the Wave specification of Section \ref{sec:Simulations} for $s \ge 0.5$, where $s=0.5$ entails an isotonic function $p$, and the choices $s > 0.5$ imply increasing degrees of non-isotonicity, also see Figure \ref{fig:p_alternatives}.
We find a conservative test size of zero for $s=0.5$ and increasing power with both, $s$ and $n$.
For the largest sample sizes, we can detect mild misspecifications with high power, showing that type II errors are barely a concern for the sample size considered in our application.

\section{Proofs and Technical Lemmas}

\renewcommand{\thetheorem}{B\arabic{theorem}}   
\renewcommand{\thelemma}{B\arabic{lemma}}   
\renewcommand{\thecorollary}{B\arabic{corollary}}   
\renewcommand{\theproposition}{B\arabic{proposition}}   
\renewcommand{\theequation}{B\arabic{equation}}   

\setcounter{theorem}{0}
\setcounter{lemma}{0}
\setcounter{corollary}{0}
\setcounter{proposition}{0}
\setcounter{equation}{0}

\begin{lemma} 
	\label{lem:hoeffding}
	Let $Y_1, \dots, Y_m$ be independent Bernoulli variables with expectations $p_1 \leq \dots \leq p_m$, and let $Z = Y_1 + \dots + Y_m$. Then for any $\delta \in (0,1)$,
	\[
	\Prob\{p_1 \leq u^\delta(Z, m)\} \geq 1 - \delta
	\quad\text{and}\quad
	\Prob\{p_m \geq \ell^\delta(Z, m)\} \geq 1 - \delta .
	\]
\end{lemma}

\begin{proof}[Proof of Lemma \ref{lem:hoeffding}]
	For the upper bound, note that $u^{\delta}(z,m)$ is increasing in $z$. If $b = \min\{z \in \{0, \dots, m\}: \, u^{\delta}(z, m) \geq p_1\}$, then $\Prob\{p_1 \leq u^{\delta}(Z, m)\} = \Prob(Z \geq b)$.
	By \citet[Example 1.A.25]{Shaked2007}, $Z$ is stochastically larger than $\tilde{Z}$ with binomial distribution with parameters $m$ and $p_1$, so $\Prob(Z \geq b) \geq \Prob(\tilde{Z} \geq b) \geq 1 - \delta$,
	where the last inequality follows from the validity of the Clopper-Pearson confidence bounds. The proof for the lower bound is similar.
\end{proof}

The proof of Theorem~\ref{thm:iso} uses standard results for isotonic least squares regression and the following inequalities of \citet[Theorem~1]{Hoeffding1963}.

\begin{lemma}
	\label{lem:hoeffding.2}
	Let $Y_1,Y_2,\ldots,Y_m$ be independent random variables with values in $[0,1]$ and expectations $p_1,p_2,\ldots,p_m$. Suppose that $q = m^{-1} \sum_{i=1}^m p_i \in (0,1)$, and set $\hat{q} = m^{-1} \sum_{i=1}^m Y_i$. Then for arbitrary $r \in [0,1]$,
	\begin{align*}
		\Prob(\hat{q} \le r)
		& \le \exp\{- m K(r,q)\} \le \exp\{- 2m (r - q)^2\}
		\ \ \text{if} \ r \le q , \\
		\Prob(\hat{q} \ge r)
		& \le \exp\{- m K(r,q)) \le \exp\{- 2m (r - q)^2\}
		\ \ \text{if} \ r \ge q ,
	\end{align*}
	where $K(r,q) := r \log(r/q) + (1 - r)\log[(1 - r)/(1 - q)]$.
\end{lemma}

\begin{corollary}
	\label{cor:hoeffding.2}
	For integers $m \ge 1$, $z \in \{0,1,\ldots,m\}$ and any number $\delta \in (0,1)$,
	\begin{align*}
		u^\delta(z,m) &
		\le \max \bigl\{ \xi \in [ \hat{q},1 ] :
		K(\hat{q},\xi) \le \log(1/\delta)/m \bigr\}
		\le \hat{q} + \sqrt{\log(1/\delta)/(2m)} , \\
		\ell^\delta(z,m) &
		\ge \min \bigl\{ \xi \in [0,\hat{q}] : K(\hat{q},\xi) \le \log(1/\delta)/m \bigr\}
		\ge \hat{q} - \sqrt{\log(1/\delta)/(2m)} ,
	\end{align*}
	where $\hat{q} = z/m$.
\end{corollary}

In addition, the proof of Theorem~\ref{thm:iso} makes use of the following proposition which is of independent interest, since it implies a more efficient method for computing the bounds of \cite{Yang2019}.

\begin{proposition}
	\label{prop:isobounds}
	For an arbitrary observation vector $\mathcal{Y} \in \real^n$, let $\hat{p} \colon [0,1] \to \real$ be an increasing function minimizing $\sum_{i=1}^n \{Y_i - \hat{p}(x_i)\}^2$. For some $\tau > 0$ and any index $i = 1, \dots, n$, let
	\[
	U_i
	= \min_{(j,k)\in \mathcal{J}\colon x_j \ge x_i}
	\Bigl( \frac{Z_{jk}^{\mathrm{iso}}}{n_{jk}}
	+ \frac{\tau}{\sqrt{n_{jk}}} \Bigr) ,
	\quad
	L_i
	= \max_{(j,k)\in \mathcal{J}\colon x_k \le x_i}
	\Bigl( \frac{Z_{jk}^{\mathrm{iso}}}{n_{jk}}
	- \frac{\tau}{\sqrt{n_{jk}}} \Bigr) .
	\]
	Then, the minimum for $U_i$ is attained at some $(j,k) \in \mathcal{J}$ such that $j = \min(s\colon x_s\ge x_i)$ and $\hat{p}(x_k) < \hat{p}(x_{k+1})$ or $k = n$. The maximum for $L_i$ is attained at some $(j,k) \in \mathcal{J}$ such that $\hat{p}(x_{j-1}) < \hat{p}(x_j)$ or $j = 1$ and $k = \max(s\colon x_s \le x_i)$.
\end{proposition}

\begin{proof}[Proof of Proposition \ref{prop:isobounds}]
	Consider the statement about $U_i$. The claim about $j$ follows from the fact that for fixed $k$, $Z_{jk}^{\mathrm{iso}}/n_{jk}$ is increasing and $n_{jk} = n - j + k$ is decreasing in $j \le k$. As to the upper index $k$, note that $U_i$ is the minimum of $u_{jk} = Z_{jk}^{\mathrm{iso}} n_{jk}^{-1} + \tau n_{jk}^{-1/2}$ over all $k \ge j = \min(s : x_s \ge x_i)$ such that $(j,k) \in \mathcal{J}$. Let $j \leq k_1 < k_2$ be indices such that $\hat{p}(x_k) = \hat{q}$ for $k = k_1+1,\dots,k_2$. Then, for $k = k_1, \dots, k_2$,
	\[
	Z_{jk}^{\mathrm{iso}}
	= Z_{jk_1}^{\mathrm{iso}} + (k - k_1) \hat{q}
	= B + n_{jk} \hat{q}
	\]
	with
	\[
	B = Z_{jk_1}^{\mathrm{iso}} - n_{jk_1} \hat{q}
	\begin{cases}
		\le 0 , \\
		= 0 & \text{if} \ \hat{p}(x_j) = \hat{q} .
	\end{cases}
	\]
	Consequently, for $k = k_1, \dots, k_2$,
	\[
	u_{jk}
	= \hat{q} + B n_{jk}^{-1} + \tau n_{jk}^{-1/2}
	\]
	is a concave function of $n_{jk}^{-1} \in [n_{jk_2}^{-1}, n_{jk_1}^{-1}]$, and it is increasing in $n_{jk}^{-1}$ if $\hat{q} = \hat{p}(x_j)$. This implies that
	\[
	u_{jk}
	\ge \begin{cases}
		\min (u_{jk_1}, u_{jk_2}) , \\
		u_{jk_2} & \text{if} \ \hat{q} = \hat{p}(x_j) .
	\end{cases}
	\]
	Consequently, the minimum of $u_{jk}$ over all $k \ge j$ is attained at some $k \ge j$ such that $\hat{p}(x_k) < \hat{p}(x_{k+1})$ or $k = n$, and this entails that $(j,k) \in \mathcal{J}$. The statement about $L_i$ follows from the one about $U_i$ when $x_1, \dots, x_n$ are replaced by $1-x_n, \dots, 1-x_1$ and $Y_1, \dots, Y_n$ by $-Y_n, \dots, -Y_1$.
\end{proof}

\begin{proof}[Proof of Theorem \ref{thm:iso}]
	Define $L_i^{\alpha,\mathrm{YB}} = L^{\alpha,\mathrm{YB}}(x_i)$, $U_i^{\alpha,\mathrm{YB}} = U^{\alpha,\mathrm{YB}}(x_i)$. The inequalities $L_i^{\alpha,\mathrm{nc}} \le L_i^{\alpha,\mathrm{raw}}$ and $U_i^{\alpha,\mathrm{raw}} \le U_i^{\alpha,\mathrm{nc}}$, as well as $L_i^{\alpha,\mathrm{YB}} \le \hat{p}(x_i) \le U_i^{\alpha,\mathrm{YB}}$ hold by construction. It is therefore sufficient to show that $L_i^{\alpha,\mathrm{YB}} \le L_i^{\alpha,\mathrm{raw}}$ and $U_i^{\alpha,\mathrm{raw}} \le U_i^{\alpha,\mathrm{YB}}$. As to the inequality $U_i^{\alpha,\mathrm{raw}} \le U_i^{\alpha,\mathrm{YB}}$, we know that $U_i^{\alpha,\mathrm{YB}}$ equals
	\[
	u_{jk}^{\mathrm{YB}}
	= Z_{jk}^{\mathrm{iso}} n_{jk}^{-1} + \tau n_{jk}^{-1/2}
	\]
	for some $(j,k) \in \mathcal{J}$ with $j = \min\{s : x_s \ge x_i\}$ and $\hat{p}(x_k) < \hat{p}(x_{k+1})$ or $k = n$, where $\tau = \sqrt{\log\{(N^2 + N)/\alpha\}/2}$. As explained later, this implies that
	\begin{equation}
		\label{eq:Z.and.Ziso}
		Z_{jk} \le Z_{jk}^{\mathrm{iso}}
		\quad\text{if} \ \hat{p}(x_k) < \hat{p}(x_{k+1}) \ \text{or} \ k = n .
	\end{equation}
	But then it follows from Corollary~\ref{cor:hoeffding.2} that $U_i^{\alpha,\mathrm{YB}} = u_{jk}^{\mathrm{YB}}$ is greater than or equal to
	\[
	Z_{jk}n_{jk}^{-1} + \tau n_{jk}^{-1/2}
	\ge u^{\alpha/(N^2+N)}(Z_{jk},n_{jk})
	\ge U_i^{\alpha,\mathrm{raw}} .
	\]
	
	Inequality \eqref{eq:Z.and.Ziso} follows from a standard result about isotonic regression \citep[see for example][Characterization II]{Henzi2020}. The index interval $\{j,\ldots,k\}$ may be partitioned into index intervals $\{\ell,\ldots,m\} = \{j,\ldots,n\} \cap \{s : \hat{p}(x_s) = \hat{q}\}$, where $\hat{q}$ is any value in $\{\hat{p}(x_j),\ldots,\hat{p}(x_k)\}$. For such an index interval, $Z_{\ell m} \le Z_{\ell m}^{\mathrm{iso}}$, with equality if $\hat{q} > \hat{p}(x_j)$.
	
	The inequality for the lower bound follows from the one for the upper bound when $x_1, \dots, x_n$ are replaced by $-x_n, \dots, -x_1$ and $Y_1, \dots, Y_n$ by $1-Y_n, \dots, 1-Y_1$.
\end{proof}

\small
\bibliographystyle{apalike}
\bibliography{CalibrationBands_V9arXiv_biblio}

\begin{thebibliography}{}

\bibitem[Allison, 2014]{Allison2014}
Allison, P.~J. (2014).
\newblock Measures of fit for logistic regression.
\newblock {\em Paper 1485-2014, SAS Global Forum 2014}, pages 1--12.

\bibitem[Bertolini et~al., 2000]{Bertolini2000}
Bertolini, G., D'Amico, R., Nardi, D., Tinazzi, A., and Apolone, G. (2000).
\newblock One model, several results: the paradox of the {H}osmer-{L}emeshow
  goodness-of-fit test for the logistic regression model.
\newblock {\em Journal of epidemiology and biostatistics}, 5:251--253.

\bibitem[Clopper and Pearson, 1934]{Clopper1934}
Clopper, C.~J. and Pearson, E.~S. (1934).
\newblock The use of confidence or fiducial limits illustrated in the case of
  the binomial.
\newblock {\em Biometrika}, 26:404--413.

\bibitem[Dimitriadis et~al., 2021]{DGJ_2021}
Dimitriadis, T., Gneiting, T., and Jordan, A.~I. (2021).
\newblock Stable reliability diagrams for probabilistic classifiers.
\newblock {\em Proceedings of the National Academy of Sciences},
  118:e2016191118.

\bibitem[D{\"u}mbgen, 1998]{Duembgen1998}
D{\"u}mbgen, L. (1998).
\newblock New goodness-of-fit tests and their application to nonparametric
  confidence sets.
\newblock {\em The Annals of Statistics}, 26:288--314.

\bibitem[Guntuboyina and Sen, 2018]{GuntuboyinaSen2018}
Guntuboyina, A. and Sen, B. (2018).
\newblock {Nonparametric shape-restricted regression}.
\newblock {\em Statistical Science}, 33(4):568--594.

\bibitem[Hall and Horowitz, 2013]{Hall2013}
Hall, P. and Horowitz, J. (2013).
\newblock A simple bootstrap method for constructing nonparametric confidence
  bands for functions.
\newblock {\em The Annals of Statistics}, 41:1892--1921.

\bibitem[Henzi et~al., 2022]{Henzi2020}
Henzi, A., Moesching, A., and D{\"u}mbgen, L. (2022+).
\newblock Accelerating the pool-adjacent-violators algorithm for isotonic
  distributional regression.
\newblock {\em Methodology and Computing in Applied Probability}.
\newblock to appear.

\bibitem[Hoeffding, 1963]{Hoeffding1963}
Hoeffding, W. (1963).
\newblock Probability inequalities for sums of bounded random variables.
\newblock {\em Journal of the American Statistical Association}, 58:13--30.

\bibitem[Hosmer and Lemeshow, 1980]{HosmerLemeshow1980}
Hosmer, D.~W. and Lemeshow, S. (1980).
\newblock Goodness of fit tests for the multiple logistic regression model.
\newblock {\em Communications in Statistics - Theory and Methods},
  9:1043--1069.

\bibitem[Hosmer et~al., 2013]{Hosmer2013Book}
Hosmer, D.~W., Lemeshow, S., and Sturdivant, R.~X. (2013).
\newblock {\em Applied logistic regression}.
\newblock Wiley Series in Probability and Statistics. Wiley, Hoboken, N.J,
  third edition.

\bibitem[Johnson et~al., 2005]{Johnson2005}
Johnson, N.~L., Kemp, A.~W., and Kotz, S. (2005).
\newblock {\em Univariate discrete distributions}.
\newblock Wiley Series in Probability and Statistics. Wiley, Hoboken, NJ, third
  edition.

\bibitem[Koenker and Yoon, 2009]{Koenker2009}
Koenker, R. and Yoon, J. (2009).
\newblock Parametric links for binary choice models: A {Fisherian}–{Bayesian}
  colloquy.
\newblock {\em Journal of Econometrics}, 152:120--130.

\bibitem[Kramer and Zimmerman, 2007]{Kramer2007}
Kramer, A.~A. and Zimmerman, J.~E. (2007).
\newblock Assessing the calibration of mortality benchmarks in critical care:
  The hosmer-lemeshow test revisited.
\newblock {\em Critical care medicine}, 35:2052--2056.

\bibitem[M{\"o}sching and D{\"u}mbgen, 2020]{Moesching2020}
M{\"o}sching, A. and D{\"u}mbgen, L. (2020).
\newblock Monotone least squares and isotonic quantiles.
\newblock {\em Electronic Journal of Statistics}, 14:24--49.

\bibitem[{National Center for Health Statistics}, 2017]{Natality2017}
{National Center for Health Statistics} (2017).
\newblock {NCHS' Vital Statistics Natality Birth Data}.
\newblock
  \href{https://data.nber.org/data/natality.html}{https://data.nber.org/data/natality.html}.
\newblock Online; accessed 13 January 2021.

\bibitem[Nattino et~al., 2014]{Nattino2014}
Nattino, G., Finazzi, S., and Bertolini, G. (2014).
\newblock A new calibration test and a reappraisal of the calibration belt for
  the assessment of prediction models based on dichotomous outcomes.
\newblock {\em Statistics in Medicine}, 33:2390--2407.

\bibitem[Nattino et~al., 2020a]{Nattino2020}
Nattino, G., Pennell, M.~L., and Lemeshow, S. (2020a).
\newblock Assessing the goodness of fit of logistic regression models in large
  samples: A modification of the hosmer-lemeshow test.
\newblock {\em Biometrics}, 76:549--560.

\bibitem[Nattino et~al., 2020b]{Nattino2020Rejoinder}
Nattino, G., Pennell, M.~L., and Lemeshow, S. (2020b).
\newblock Rejoinder to “assessing the goodness of fit of logistic regression
  models in large samples: A modification of the hosmer-lemeshow test”.
\newblock {\em Biometrics}, 76:575--577.

\bibitem[Paul et~al., 2013]{Paul2013}
Paul, P., Pennell, M.~L., and Lemeshow, S. (2013).
\newblock Standardizing the power of the {Hosmer}–{Lemeshow} goodness of fit
  test in large data sets.
\newblock {\em Statistics in Medicine}, 32:67--80.

\bibitem[Quinn et~al., 2016]{Quinn2016PretermDef}
Quinn, J.-A., Munoz, F.~M., Gonik, B., Frau, L., Cutland, C., Mallett-Moore,
  T., Kissou, A., Wittke, F., Das, M., Nunes, T., Pye, S., Watson, W., Ramos,
  A.-M.~A., Cordero, J.~F., Huang, W.-T., Kochhar, S., Buttery, J., and
  {Brighton Collaboration Preterm Birth Working Group} (2016).
\newblock Preterm birth: Case definition \& guidelines for data collection,
  analysis, and presentation of immunisation safety data.
\newblock {\em Vaccine}, 34(49):6047--6056.

\bibitem[{R Core Team}, 2022]{R2022}
{R Core Team} (2022).
\newblock {\em R: A language and environment for statistical computing}.
\newblock R Foundation for Statistical Computing, Vienna, Austria.

\bibitem[Roelofs et~al., 2020]{Roelofs2020}
Roelofs, R., Cain, N., Shlens, J., and Mozer, M.~C. (2020).
\newblock Mitigating bias in calibration error estimation.
\newblock {\em Preprint}.
\newblock
  \href{https://arxiv.org/abs/2012.08668}{https://arxiv.org/abs/2012.08668}.

\bibitem[Sen et~al., 2010]{Sen2010}
Sen, B., Banerjee, M., and Woodroofe, M. (2010).
\newblock {Inconsistency of bootstrap: The Grenander estimator}.
\newblock {\em The Annals of Statistics}, 38(4):1953--1977.

\bibitem[Shaked and Shanthikumar, 2007]{Shaked2007}
Shaked, M. and Shanthikumar, J.~G. (2007).
\newblock {\em Stochastic orders}.
\newblock Springer Series in Statistics. Springer, New York.

\bibitem[Stodden et~al., 2016]{Stodden2016}
Stodden, V., McNutt, M., Bailey, D.~H., Deelman, E., Gil, Y., Hanson, B.,
  Heroux, M.~A., Ioannidis, J.~P., and Taufer, M. (2016).
\newblock Enhancing reproducibility for computational methods.
\newblock {\em Science}, 354(6317):1240--1241.

\bibitem[Tutz, 2011]{Tutz2011}
Tutz, G. (2011).
\newblock {\em Regression for Categorical Data}.
\newblock Cambridge University Press, Cambridge.

\bibitem[{World Health Organization}, 2015]{def_who}
{World Health Organization} (2015).
\newblock {\em International statistical classification of diseases and related
  health problems}.
\newblock World Health Organization.
\newblock 10th revision, fifth edition.
  \href{https://apps.who.int/iris/handle/10665/246208}{https://apps.who.int/iris/handle/10665/246208}.
  Online; accessed 13 January 2021.

\bibitem[Wright, 1981]{Wright1981}
Wright, F.~T. (1981).
\newblock The asymptotic behavior of monotone regression estimates.
\newblock {\em Annals of Statistics}, 9:443--448.

\bibitem[Yang and Barber, 2019]{Yang2019}
Yang, F. and Barber, R.~F. (2019).
\newblock Contraction and uniform convergence of isotonic regression.
\newblock {\em Electronic Journal of Statistics}, 13:646--677.

\bibitem[Yu and Kumbier, 2020]{Yu2020}
Yu, B. and Kumbier, K. (2020).
\newblock Veridical data science.
\newblock {\em Proceedings of the National Academy of Sciences},
  117(8):3920--3929.

\end{thebibliography}

\newpage
\setcounter{page}{1}
\setcounter{footnote}{0}
\begin{center}
	SUPPLEMENTARY MATERIAL FOR   \vspace{10pt} \\
	{\Large\bf {Honest calibration assessment for binary outcome predictions}}\\[10pt]
	Timo Dimitriadis, Lutz D\"umbgen, Alexander Henzi, Marius Puke and Johanna Ziegel \\[5pt]
	\today \\ 
\end{center}


\renewcommand{\thesection}{S.\arabic{section}}   
\renewcommand{\thepage}{S.\arabic{page}}  
\renewcommand{\thetable}{S\arabic{table}}   
\renewcommand{\thefigure}{S\arabic{figure}}   
\renewcommand{\theequation}{S\arabic{equation}}   
\renewcommand{\thelemma}{S\arabic{lemma}}   




\setcounter{section}{0}
\setcounter{table}{0}
\setcounter{figure}{0}
\setcounter{equation}{0}
\setcounter{lemma}{0}

The Supplementary Materials contains four parts. 
Section \ref{sec:RoundingDetails} demonstrates the effect of using a restricted set of index intervals.
Section \ref{sec:ModelSpecApplication} gives details on the regression model specifications in the low birth weight application.
Section \ref{sec:YBBounds} illustrates the gains of our method upon the wider bands of \cite{Yang2019} in this application.
Section \ref{sec:AddProofs} gives additional proofs.

\section{The effect of using a restricted  family of index intervals}
\label{sec:RoundingDetails}

As discussed before equation \eqref{eqn:Rounding} in the main manuscript and informally described as the rounding method, the confidence bands in equations \eqref{eq:upper_bound} and \eqref{eq:lower_bound} also achieve correct coverage in the sense of \eqref{eq:coverage} if we only consider a restricted family of index pairs $\widetilde{\mathcal{J}} \subset \mathcal{J}$.
Besides the reduced computation time, which we discuss below, this has the additional advantage that it reduces the correction factor of the significance level from $|\mathcal{J}| = N^2 + N$ to $|\widetilde{\mathcal{J}}|$.
However, the optimal index interval as selected by the infimum in \eqref{eq:upper_bound} and the supremum in \eqref{eq:lower_bound}  over the full set $\mathcal{J}$ may not be contained in $\widetilde{\mathcal{J}}$, resulting in a possibly wider confidence band.
While a general balancing of these two opposing effects is difficult without knowledge of the true form of $p$, Figure \ref{fig:rounding} illustrates the effect of the rounding method with the explicit choice of $\widetilde{\mathcal{J}}$ in  \eqref{eqn:Rounding}  based the choices $K \in \{20, 100, \infty \}$ on simulated data.

First assume that the curve $p$ is (almost) flat.
Then, the infimum in the computation of $U^{\alpha, \text{raw}}(x)$ in equation  \eqref{eq:upper_bound} is most likely attained for the largest index interval in $\mathcal{J}$, i.e., by computing the Clopper-Pearson confidence bounds using all indices $x_j \ge x$.
Hence, as long as $\widetilde{\mathcal{J}}$ in equation \eqref{eqn:Rounding} approximately contains this full index interval, there is almost no effect of the rounding in terms of an inefficient selection of the index intervals.
However, as the correction factor of the significance level is reduced from $N^2 + N$ to $|\widetilde{\mathcal{J}}|$, this entails thinner intervals as can be seen in the region $x \ge 0.3$ in Figure \ref{fig:rounding}.

\begin{figure}[tb]
	\centering
	\includegraphics[width=\textwidth]{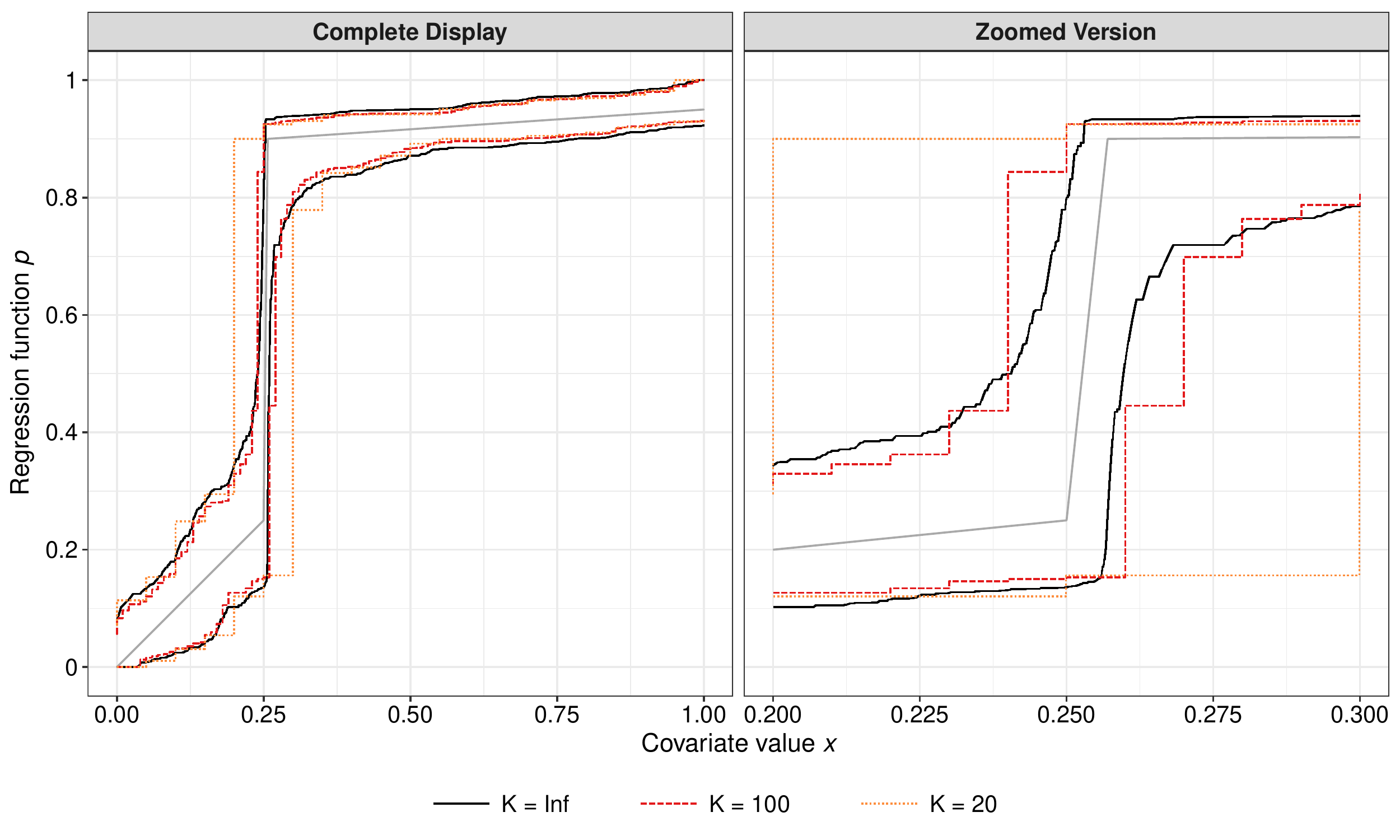}
	\caption{
		This figure illustrates the effect of the rounding method in \eqref{eqn:Rounding} with the choices $K \in \{20, 100, \infty \}$, the latter corresponding to no rounding. We simulate $n=10\,000$ data points with a true regression function $p(\cdot)$ that linearly connects the points $(0,0)$, $(0.25,0.25)$, $(0.275,0.9)$ and $(1,0.95)$.
		The left plot shows the true regression function in gray together with the confidence bands and the plot on the right side is a magnified version  focusing on the steep area of the regression function.
	}
	\label{fig:rounding}
\end{figure}

In contrast, in steeper regions of $p$, the inefficient index interval selection mechanism stemming from a restricted $\widetilde{\mathcal{J}}$ might have a bigger adverse effect than the lower correction factor of the significance level.
This effect can be observed in the particularly steep region around $x=0.25$ in the zoomed version of the plot in the right side of Figure \ref{fig:rounding}, where the choice $K=\infty$ yields the most narrow bands.
Finally, the region with $x \le 0.25$ having unit slope (pertaining to the most important case of perfectly calibrated predictions in applications on calibration assessment) shows that rounding with $K=100$ improves the bands whereas further reducing $K$ results in too coarse approximations, also limiting the adaptivity of the band derived in Theorem \ref{thm:asymptotics}.

Furthermore, the choice of $K$ in $\widetilde{\mathcal{J}}$ massively affects the computation times required for the bands.
Figure \ref{fig:computation_time} displays the required computation time to compute the infimum and supremum in equations  \eqref{eq:upper_bound} and \eqref{eq:lower_bound} for the full index set $\mathcal{J}$, and two reduced sets $\widetilde{\mathcal{J}}$ with $K=100$ and $K=1000$ together with the computation time of the \cite{Yang2019} bands.

\begin{figure}[tb]
	\centering
	\includegraphics[width=\textwidth]{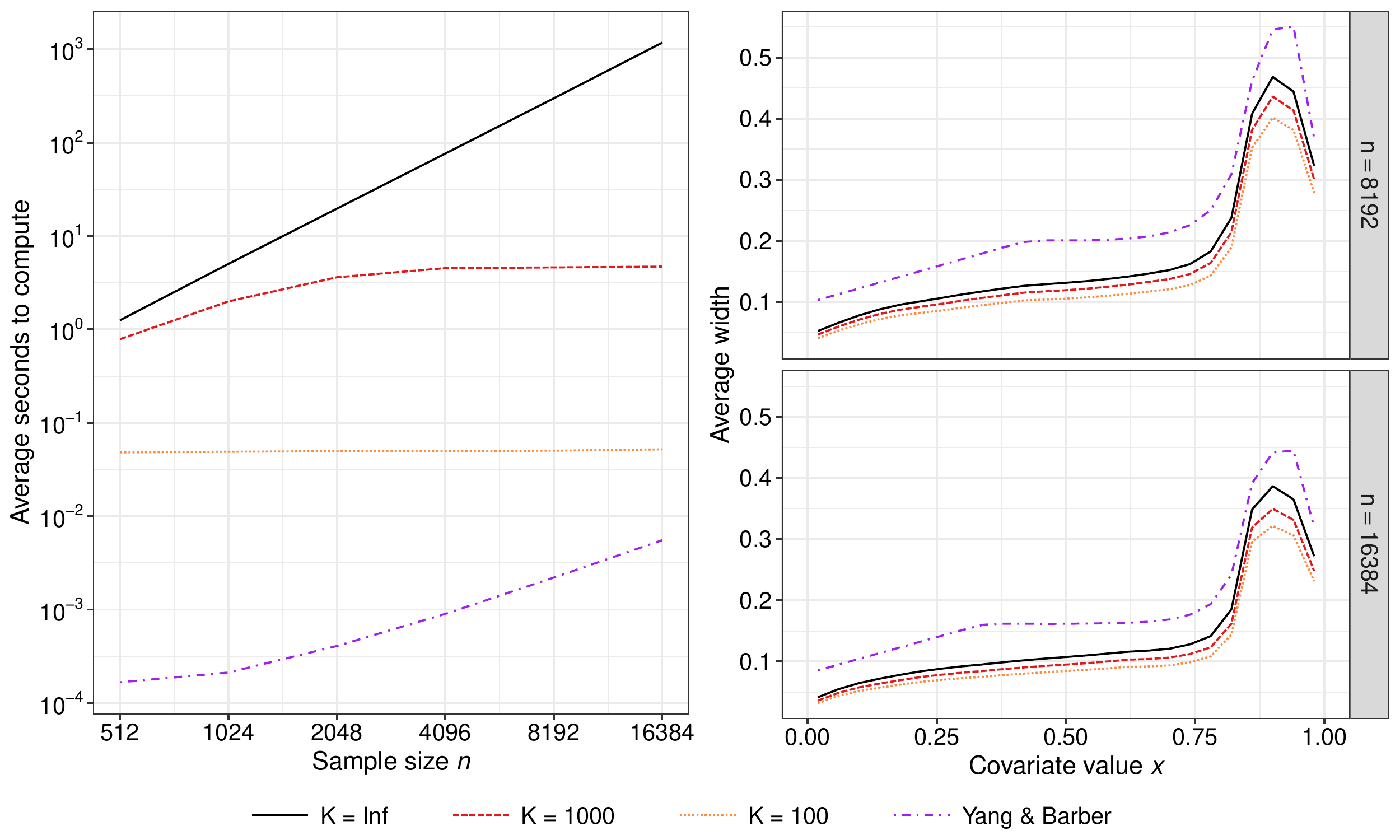}
	\caption{
		The left plot shows the computation time in simulations of the Kink process with $s=0.8$ from Section \ref{sec:Simulations} of the main manuscript  depending on the sample size for a range of rounding values $K$ and the bands of \citet{Yang2019}. We use a logarithmic scale for both axes.
		The two plots on the right show the average width of these four versions of the confidence bands depending on the covariate value $x$. 
	}
	\label{fig:computation_time}
\end{figure}

As expected, we find that the computation time of our standard method grows at rate $n^2$, where computing the bands takes up to 15 minutes for $n=16\,384$.
The computation time decreases drastically for both rounding methods, and even for $K=1000$, the bands for $n=16\,384$ are computed within five seconds.
For an increasing sample size, the computation time stabilizes once the size of $\widetilde{\mathcal{J}}$ stays constant.
Finally, the bands of \cite{Yang2019} have the lowest computation time throughout all considered sample sizes as it suffices to take the minimum over endpoints of constancy regions of the isotonic regression estimate for these bands, which is explained in the end of Section \ref{sec:Relation_to_YB} of the main manuscript.
The display of the average width on the right-hand side of Figure \ref{fig:computation_time} confirms that medium values of e.g., $K=100$ or $K=1000$ yield relatively narrow bands.

Summarizing the results of this section, the rounding method can drastically decrease the computation time and even results in narrower bands for all but very steep regions of the regression function.

\section{Model specifications in the low birth weight application}
\label{sec:ModelSpecApplication}

We give some additional details on the model specifications of the application here.
The first two models are based on the probit link function whereas the third one uses the cauchit link function \citep{Koenker2009}.
The second model uses the week of gestation as a continuous variable whereas the first and third models use the week of gestation as a categorical variable with left-closed and right-open intervals with lower interval limits of 0, 28, 32 and 37 weeks, which corresponds to the standard categorization of the World Health Organization \citep{Quinn2016PretermDef}.

Additionally, all three models contain the following common explanatory variables: 
the mother's age and its squared term, her body mass index prior to pregnancy, her smoking behavior as a categorical variable with left-closed and right-open intervals with lower limits of 0, 1, 9, and 20 cigarettes per day averaged over all three trimesters, individual binary variables for mother's diabetes, any form of hypertension, mother's education below or equal to eight years, employed infertility treatments, a cesarean in a previous pregnancy, a preterm birth in a previous pregnancy, current multiple pregnancy, the sex of the unborn child, and an infection of one of the following: gonorrhea, syphilis, chlamydia, hepatitis b, hepatitis c.
Additional details on the data are given in the user guide under \href{https://data.nber.org/natality/2017/natl2017.pdf}{https://data.nber.org/natality/2017/natl2017.pdf}.

\section{The Yang and Barber bands in the low birth weight application}
\label{sec:YBBounds}

Figure \ref{fig:appl_lbw_AS} illustrates the bands of \cite{Yang2019} with a minimal variance factor of $\sigma^2 = 1/4$ in the three binary regression specifications presented in Figures \ref{fig:ApplicationIntro} and \ref{fig:appl_lbw} of the main manuscript.
We see that these bands are substantially wider than ours, especially in the most important region of small probability predictions, e.g., illustrated in the zoomed version in the upper right panel of the figure.
This improvement is theoretically explained by Theorem \ref{thm:asymptotics} (iv) and the corresponding discussion thereafter:
Our confidence bands adapt to the variance of the observation, i.e., their width is smaller for $p$ close to zero or one as compared to $p$ around $0.5$.

\begin{figure}[tb]
	\centering
	\includegraphics[width=\textwidth]{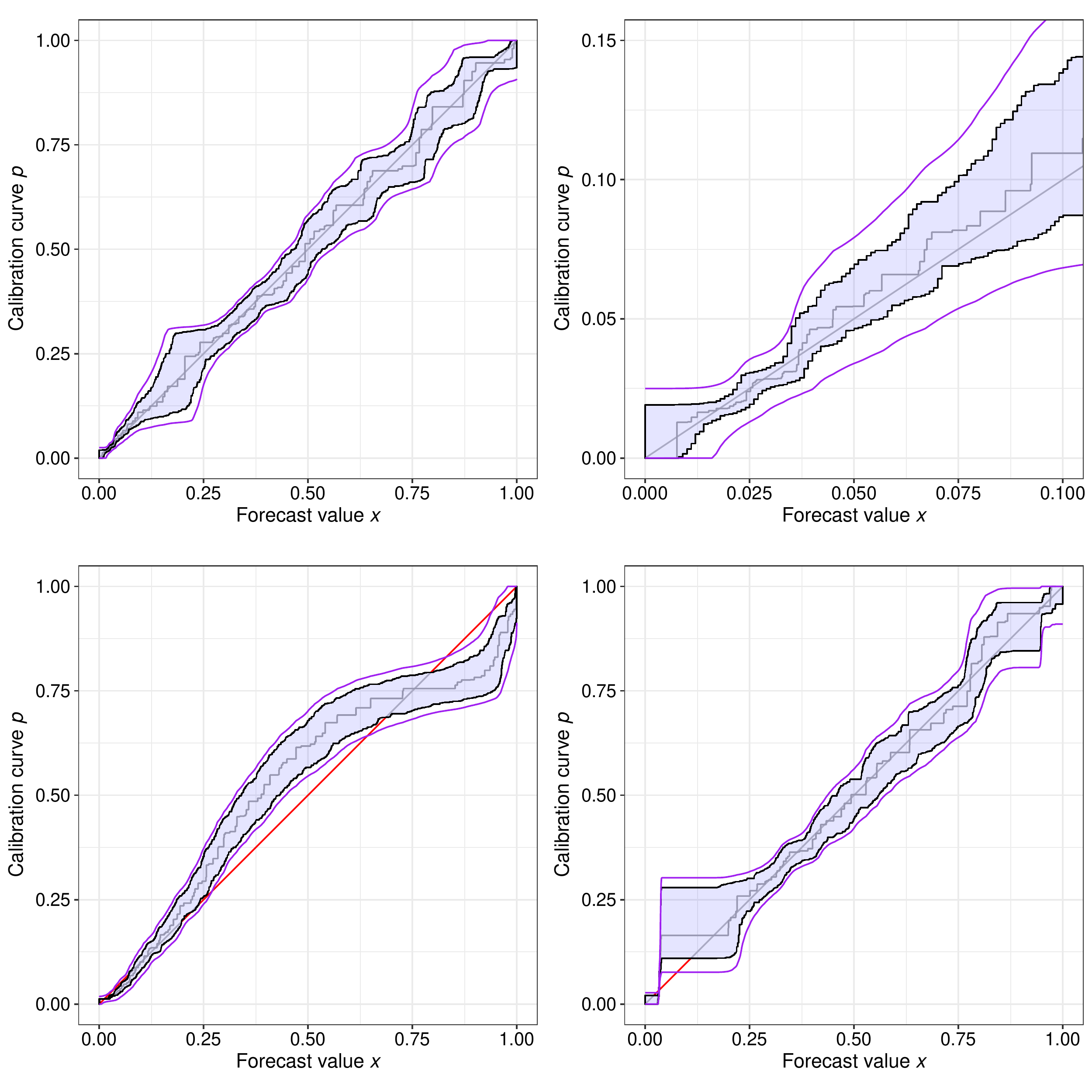}
	\caption{This figures shows the plots given in Figures  \ref{fig:ApplicationIntro} and \ref{fig:appl_lbw}  augmented with the bands of \citet{Yang2019} in purple color.}
	\label{fig:appl_lbw_AS}
\end{figure}

\section{Additional Proofs}
\label{sec:AddProofs}

\begin{proof}[Proof of Lemma \ref{lem:isopart}]
	Let $\hat{q}_1 < \cdots < \hat{q}_b$ be the different elements of $\{\hat{p}(x_i)\colon 1 \le i \le n\}$, where we assume that $b \ge 2$. There exists a partition of $\{1,\ldots,n\}$ into index intervals $I_1,\ldots,I_b$ such that $\hat{q}_\ell = |I_\ell|^{-1} \sum_{i \in I_\ell} Y_i$. For any integer $d \ge 1$, let $M_d$ be the number of indices $\ell$ such that $|I_\ell| = d$. Since $\sum_{i\in I_\ell} Y_i \in \{0,1,\ldots,d\}$, the numbers $M_d$ satisfy the following constraints: $M_d \in [0,d+1]$, and $\sum_{d=1}^n M_d d = n$. The question is, how large the number $b = \sum_{d=1}^n M_d$ can be under these constraints, where we drop the restriction that the $M_d$ are integers. Suppose that $M_c < c+1$ and $M_{c'} > 0$ for integers $1 \le c < c'$. Then we may replace $(M_c,M_{c'})$ with $(M_c + \gamma/c, M_{c'} - \gamma/c')$, where $\gamma$ is the minimum of $(c+1-M_c)c$ and $M_{c'} c'$. This does not affect the constraints, but the sum $\sum_{d=1}^n M_d$ increases strictly, while $M_c = c+1$ or $M_{c'} = 0$. Eventually, we obtain an integer $d_o \ge 1$ such that $M_d = d+1$ if $1 \le d \le d_o$ and $M_d = 0$ for $d \ge d_o + 2$. In particular,
	\[
	n \ge \sum_{d=1}^{d_o} (d+1)d = \frac{(d_o+2)(d_o+1)d_o}{3} > \frac{d_o^3}{3} ,
	\]
	whence $d_o < (3n)^{1/3}$, while
	\[
	b \le \sum_{d=1}^{d_o+1} (d+1) = \frac{d_o(d_o+3)}{2} \le C n^{2/3} ,
	\]
	where $C = 3^{2/3}(1 + 3/6^{1/3})/2 < 3$.
\end{proof}

For the proof of Theorem~\ref{thm:asymptotics}, we need an inequality for the auxiliary function $K(\cdot,\cdot)$ in Lemma~\ref{lem:hoeffding.2} which follows from \citet[Proposition~2.1]{Duembgen1998}.

\begin{lemma}
	\label{lem:bounds.for.K}
	For arbitrary $q \in [0,1]$, $\xi \in (0,1)$ and $\gamma > 0$, the inequality $K(q,\xi) \le \gamma$ implies that
	\[
	|\xi - q|
	\le \sqrt{2\gamma q(1 - q)} + |1 - 2q| \gamma .
	\]
\end{lemma}

\begin{proof}[Proof of Theorem \ref{thm:asymptotics}]
	For notational convenience, we often drop the additional subscript $n$, e.g.\ we write $x_i$ instead of $x_{ni}$. For symmetry reasons, it suffices to verify the assertions about $U^{\alpha,\mathrm{raw}}$. We only consider sample sizes $n$ such that the inequalities for $W_n(B)$ in Assumption~\ref{ass:A} are valid.
	
	In what follows, let $C$ be a generic (large) constant depending only on $C_1, C_2$. Its value may change in each instance. It follows from Corollary~\ref{cor:hoeffding.2} and Lemma~\ref{lem:bounds.for.K} that for sufficiently large $n$, simultaneously for all $(j,k) \in \mathcal{J}$,
	\begin{equation}
		\label{ineq:ujk}
		u^{\alpha/(N^2+N)}(Z_{jk},n_{jk})
		\le \hat{p}_{jk} + C \min \biggl\{
		\sqrt{\frac{\log(n) \hat{p}_{jk}}{n_{jk}}}
		+ \frac{\log(n)}{n_{jk}} ,
		\sqrt{\frac{\log(n)}{n_{jk}}} \biggr\} ,
	\end{equation}
	where $\hat{p}_{jk} = Z_{jk}/n_{jk}$. Note that we got rid of $\alpha$, because $\log\{(N^2 + N)/\alpha\} \le \log\{(n^2 + n)/\alpha\} = 2 \log(n) (1 + o(1))$ as $n \to \infty$. Moreover, one can deduce from Lemma~\ref{lem:hoeffding.2} that simultaneously for all $(j,k) \in \mathcal{J}$,
	\begin{equation}
		\label{ineq:phatjk}
		\hat{p}_{jk}
		\le p_{jk} + C \sqrt{\frac{\log(n)}{n_{jk}}}
	\end{equation}
	with asymptotic probability one, where $p_{jk} = \mathbb{E}(\hat{p}_{jk}) = n_{jk}^{-1} \sum_{i=j}^k p_i \in [p_j,p_k]$.
	
	As to part~(i), let $x \in [a_o,b_o)$ and $B(x) = [x,b_o]$. If $x \le b_o - C_2 \rho_n$, then it follows from Assumption~\ref{ass:A}  that $B(x) \cap \{x_1,\ldots,x_n\} = \{x_{j(x)},\ldots,x_{k(x)}\}$ with $(j(x),k(x)) \in \mathcal{J}$ such that
	\[
	n_{j(x)k(x)} = W_n\{B(x)\} \ge C_1 n (b_o - x) .
	\]
	Consequently, we may deduce from inequalities \eqref{ineq:ujk} and \eqref{ineq:phatjk} that with asymptotic probability one, simultaneously for all $x \in [a_o,b_o-C_2\rho_n]$,
	\[
	U_n^{\alpha,\mathrm{raw}}(x)
	\le u^{\alpha/(N^2+N)}(Z_{j(x)k(x)},n_{j(x)k(x)})
	\le \hat{p}_{j(x)k(x)} + C \sqrt{\rho_n/(b_o - x)} ,
	\]
	\[
	\hat{p}_{j(x)k(x)}
	\le p_{j(x),k(x)} + C \sqrt{\rho_n/(b_o - x)}
	= p(x) + C \sqrt{\rho_n/(b_o - x)} .
	\]
	These two inequalities imply that $U_n^{\alpha,\mathrm{raw}}(x) \le p(x) + C \sqrt{\rho_n/(b_o - x)}$ for $x \in [a_o, b_o - C_2\rho_n]$. But for $x \in [b_o - C_2\rho_n, b_o)$, the term $\sqrt{\rho_n/(b_o - x)}$ is at least $C_2^{-1/2}$, and $U_n^{\alpha,\mathrm{raw}}(x) - p(x) \le 1$. Hence we can deduce part~(i) by replacing $C$ with $\max\{C,C_2^{1/2}\}$.
	
	As to part~(ii), let $B(x) = [x,x+h_n]$ for $x \in [a_o, b_o-h_n]$ with some constant $h_n \ge C_2\rho_n$ to be determined later. By Assumption~\ref{ass:A}, $B(x) \cap \{x_1,\ldots,x_n\} = \{x_{j(x)},\ldots,x_{k(x)}\}$ with $(j(x),k(x)) \in \mathcal{J}$ satisfying
	\[
	n_{j(x)k(x)} = W_n\{B(x)\} \ge C_1 n h_n .
	\]
	Consequently, we may deduce from inequalities \eqref{ineq:ujk}, \eqref{ineq:phatjk} and Lipschitz-continuity of $p$ on $[a_o,b_o]$ with Lipschitz constant $L$ that with asymptotic probability one, simultaneously for all $x \in [a_o,b_o-h_n]$,
	\begin{align*}
		U_n^{\alpha,\mathrm{raw}}(x)
		\le u^{\alpha/(N^2+N)}(Z_{j(x)k(x)},n_{j(x)k(x)})
		& \le \hat{p}_{j(x)k(x)} + C \sqrt{\rho_n/h_n} , \\
		\hat{p}_{j(x)k(x)}
		& \le p_{j(x),k(x)} + C \sqrt{\rho_n/h_n} , \\
		p_{j(x)k(x)}
		& \le p(x) + L h_n .
	\end{align*}
	These three inequalities imply that $U_n^{\alpha,\mathrm{raw}}(x) \le p(x) + C \sqrt{\rho_n/h_n} + L h_n$. If we set $h_n = \rho_n^{1/3} L^{-2/3}$, the upper bound becomes $C (L\rho_n)^{1/3}$. This requires $\rho_n^{1/3} L^{-2/3} \ge C_2 \rho_n$, though. But in case of $\rho_n^{1/3} L^{-2/3} \le C_2 \rho_n$, the term $(L \rho_n)^{1/3}$ is at least $C_2^{-1/2}$, so we can deduce part~(ii) by replacing $C$ with $\max\{C, C_2^{1/2}\}$.
	
	Part~(iii) can be verified similarly as part~(i). let $B(x) = [x,x_o)$ for $x \in [a_o,x_o)$. In case of $x \le x_o - C_2\rho_n$, $B(x) \cap \{x_1,\ldots,x_n\} = \{x_{j(x)},\ldots,x_{k(x)}\}$ with $(j(x),k(x)) \in \mathcal{J}$ such that $n_{j(x)k(x)} \ge C_1 n (x_o - x)$. Thus it follows from inequalities \eqref{ineq:ujk} and \eqref{ineq:phatjk} that with asymptotic probability one, simultaneously for all $x \in [a_o,x_o-C_2\rho_n]$,
	\[
	U_n^{\alpha,\mathrm{raw}}(x)
	\le u^{\alpha/(N^2+N)}(Z_{j(x)k(x)},n_{j(x)k(x)})
	\le \hat{p}_{j(x)k(x)} + C \sqrt{\rho_n/(b_o - x)} ,
	\]
	\[
	\hat{p}_{j(x)k(x)}
	\le p_{j(x),k(x)} + C \sqrt{\rho_n/(b_o - x)}
	\le p(x_o-) + C \sqrt{\rho_n/(b_o - x)} .
	\]
	If $x \in [x_o - C_2\rho_n, x_o)$, the term $\sqrt{\rho_n/(b_o - x)}$ is at least $C_2^{-1/2}$, so we can deduce part~(iii) be replacing $C$ with $\max\{C,C_2^{1/2}\}$.
	
	To verify part~(iv), let $B(x,y) = [x,y]$ for $a_o \le x < y \le b_o$. If $y - x \ge C_2 \rho_n$, then $B(x,y) \cap \{x_1,\ldots,x_n\} = \{x_{j(x,y)},\ldots,x_{k(x,y)}\}$ with $(j(x,y),k(x,y)) \in \mathcal{J}$ such that $n_{j(x,y)k(x,y)} \ge C_1 n (y - x)$. Hence, it follows from \eqref{ineq:ujk} and $p_{j(x,y)k(x,y)} \le p(y)$ that
	\begin{align*}
		\operatorname{\mathbb{E}} \bigl\{U_n^{\alpha,\mathrm{raw}}(x_n)\bigr\}
		&\le \operatorname{\mathbb{E}} \bigl\{
		u^{\alpha/(N^2 + N)}(Z_{j(x,y)k(x,y)},n_{j(x,y)k(x,y)})\bigl\} \\
		&\le \operatorname{\mathbb{E}} \Bigl\{ \hat{p}_{j(x,y)k(x,y)}
		+ C \bigl( \sqrt{\hat{p}_{j(x,y)k(x,y)} \rho_n/(y - x)} + \rho_n/(y - x) \bigr)
		\Bigr\} \\
		&\le p_{j(x,y)k(x,y)}
		+ C \bigl( \sqrt{ p_{j(x,y)k(x,y)} \rho_n/(y - x)} + \rho_n/(y - x) \bigr) \\
		&\le p(y) + C \bigl( \sqrt{ p(y) \rho_n/(y - x)} + \rho_n/(y - x) \bigr) \\
		&\le C \{ p(y) + \rho_n/(y - x) \} ,
	\end{align*}
	where the third inequality follows from Jensen's inequality, and the last inequality follows from $\sqrt{st} \le (s + t)/2$ for $s,t \ge 0$. This is true if $y - x \ge C_2 \rho_n$. But in case of $y - x \le C_2 \rho_n$, the term $\rho_n/(y - x)$ is at least $C_2^{-1}$, so we can deduce part~(iv) by replacing $C$ with $\max\{C,C_2\}$.
\end{proof}

\end{document}